\RequirePackage[OT1]{fontenc}
\documentclass[12pt,draftcls,onecolumn]{IEEEtran}

\usepackage{subfigure}
\usepackage{amssymb, amsmath, amsfonts}
\usepackage{empheq}
\usepackage{stfloats}
\usepackage{caption, cite}
\usepackage{graphicx}

\usepackage{algorithm}
\usepackage{algorithmic}

\usepackage{amsthm}
\usepackage{tikz}
\usepackage{amsmath}
\allowdisplaybreaks
\usetikzlibrary{shadows,arrows,positioning}
\pgfdeclarelayer{background}
\pgfdeclarelayer{foreground}
\pgfsetlayers{background,main,foreground}

\usetikzlibrary{external}
\usepackage{pgfplots}
\usepackage{cases}

\newtheorem{proposition}{Proposition}
\newtheorem{theorem}{Theorem}
\newtheorem{definition}{Definition}
\newtheorem{lemma}{Lemma}

\newtheorem{corollary}{Corollary}
\newtheorem{remark}{Remark}
\newtheorem{assumption}{Assumption}

\newlength\figureheight
\newlength\figurewidth

\DeclareFontFamily{OT1}{pzc}{}
\DeclareFontShape{OT1}{pzc}{m}{it}{<-> s * [1.000] pzcmi7t}{}
\DeclareMathAlphabet{\mathpzc}{OT1}{pzc}{m}{it}

\newcommand{\R}{{\mathbb{R}}}
\newcommand{\E}{{\mathbb{E}}}

\newcommand{\vv}{{\mathbf{v}}}

\newcommand{\bv}{{\mathbf{\bar{v}}}}
\newcommand{\x}{{\mathbf{x}}}
\newcommand{\g}{{\mathbf{g}}}
\newcommand{\bg}{{\mathbf{\bar{g}}}}
\newcommand{\bx}{{\mathbf{\bar{x}}}}
\newcommand{\hx}{{\mathbf{\hat{x}}}}
\newcommand{\sxk}{{\mathbf{\sigma}_{x,k}}}
\newcommand{\y}{{\mathbf{y}}}
\newcommand{\hy}{{\mathbf{\hat{y}}}}
\newcommand{\syk}{{\mathbf{\sigma}_{y,k}}}
\newcommand{\by}{{\mathbf{\bar{\y}}}}
\newcommand{\f}{{\nabla{\mathbf{f}}}}
\newcommand{\bbf}{{\nabla{\mathbf{\bar{f}}}}}
\newcommand{\lf}{{\nabla{f}}}
\newcommand{\oxk}{{\Omega_{x,k}}}
\newcommand{\oyk}{{\Omega_{y,k}}}
\newcommand{\osxk}{{\Omega_{\sigma_{x},k}}}
\newcommand{\osyk}{{\Omega_{\sigma_{y},k}}}
\newcommand{\oo}{{\Omega}}

\newcommand{\lb}{{\bar{\lambda}}}

\newcommand{\llw}{{\rho_{w}}}
\newcommand{\lwi}{{\bar{\lambda}_{W-I}}}
\newcommand{\nxk}{{\mathbf{\xi}_{x,k}}}
\newcommand{\nxkk}{{\mathbf{\xi}_{x,k+1}}}
\newcommand{\nbxk}{{\mathbf{\bar{\xi}}_{x,k}}}

\newcommand{\nyk}{{\mathbf{\xi}_{y,k}}}
\newcommand{\nykk}{{\mathbf{\xi}_{y,k+1}}}
\newcommand{\nvk}{{\mathbf{\xi}_{v,k}}}

\newcommand{\nbyk}{{\mathbf{\bar{\xi}}_{y,k}}}

\newcommand{\bl}{{\mathbf{L}}}
\newcommand{\bi}{{\mathbf{I}}}
\newcommand{\bk}{{\mathbf{K}}}
\newcommand{\bp}{{\mathbf{P}}}
\newcommand{\bh}{{\mathbf{H}}}

\newcommand\addtag{\refstepcounter{equation}\tag{\theequation}}

\makeatletter

\newcommand{\Rmnum}[1]{\expandafter\@slowromancap\romannumeral #1@}
\makeatother

\title{ \hspace*{\fill} \\\hspace*{\fill} \\ \LARGE \bf{Differentially Private and Communication-Efficient Distributed Nonconvex Optimization Algorithms}}

\author{Antai Xie, Xinlei Yi, Xiaofan Wang, Ming Cao, and Xiaoqiang Ren
\thanks{A. Xie, X. Wang, and X. Ren are with the School of Mechatronic Engineering and Automation, Shanghai University, Shanghai, China. Emails: \{xatai,\,xfwang,\,xqren\}@shu.edu.cn.}
\thanks{X. Yi is with the Lab for Information \& Decision Systems, Massachusetts Institute of Technology, Cambridge, MA 02139, USA. Email: xinleiyi@mit.edu.}
\thanks{M. Cao is with the Faculty of Science and Engineering, University of Groningen, Groningen, the Netherlands. Email: m.cao@rug.nl.}
}
\begin{document}
	\maketitle
	 \begin{abstract}
This paper studies the privacy-preserving distributed optimization problem under limited communication, where each agent aims to keep its cost function private while minimizing the sum of all agents’ cost functions. To this end, we propose two differentially private distributed algorithms under compressed communication. We show that the proposed algorithms achieve sublinear convergence for smooth (possibly nonconvex) cost functions and linear convergence when the global cost function additionally satisfies the Polyak–Łojasiewicz condition, even for a general class of compressors with bounded relative compression error. Furthermore, we rigorously prove that the proposed algorithms ensure $\epsilon$-differential privacy. Unlike methods in the literature, the analysis of privacy under the proposed algorithms do not rely on the specific forms of compressors. Simulations are presented to demonstrate the effectiveness of our proposed approach.
	\end{abstract}
	
\section{Introduction}
In recent years, distributed optimization in multi-agent systems has emerged as a popular research topic, playing a fundamental role in areas such as resource allocation~\cite{xu2017distributed}, control~\cite{nedic2018distributed}, learning~\cite{li2020distributed}, and estimation~\cite{cattivelli2009diffusion}. In a typical distributed consensus optimization setup, the objective is for a team of agents connected through a network, each associated with a local cost function, to cooperatively minimize the sum of local cost functions. Specifically, consider a network of $n$ agents aiming to solve the following optimization problem:
\begin{align}\label{P1}
	\min_{x\in\mathbb{R}^d}\left\{f(x)=\frac{1}{n}\sum_{i=1}^n f_i(x)\right\},
\end{align}
where $f_i:\mathbb{R}^d\mapsto\mathbb{R}$ is private local cost function belong to agent~$i$ and $x$ is the global decision variable.

Numerous distributed optimization algorithms have been reported to solve the problem~\eqref{P1}, such as distributed (sub)gradient descent~\cite{nedic2009distributed,xu2017convergence,yuan2016convergence}, gradient tracking methods~\cite{qu2017harnessing}, EXTRA~\cite{shi2015extra}, and distributed Newton methods~\cite{varagnolo2015newton,wei2013distributed}. However, these algorithms usually assume that the cost functions $f_i$ are convex. In many applications, such as empirical risk minimization~\cite{bottou2018optimization} and resource allocation~\cite{tychogiorgos2013non}, the cost functions are nonconvex. To address this issue, the authors of~\cite{zeng2018nonconvex,necoara2019linear,wai2017decentralized} proposed several distributed algorithms that allow each agent to achieve the first-order stationary point even in the presence of nonconvex cost functions.
	 
To implement the distributed algorithms, agents need to communicate with each other, which is normally realized by wireless networks. However, wireless networks are vulnerable to malicious attacks, which can result in eavesdropping of the sensitive information transmitted between agents. For instance, in the robot rendezvous problem, the decision variables may contain some private and sensitive location information, as highlighted in~\cite{zhang2018admm}. Moreover, recent research~\cite{zhu2019deep} has shown that adversaries can recover private training data through shared gradients, leading to the risk of exposing confidential information such as medical records and financial transactions. It is therefore essential to promptly and comprehensively address privacy concerns in distributed optimization.

To ensure the privacy of each agent in distributed optimization, various privacy-preserving algorithms have been proposed. Among them, there are two common categories of algorithms. The first category involves adding noise to the transmitted information to confuse attackers. The authors of~\cite{huang2015differentially, zhu2018differentially, ding2021differentially, chen2023differentially} proposed several distributed optimization algorithms making use of the notion of differential privacy~\cite{dwork2008differential}. For instance, Huang \textit{et al.}~\cite{huang2015differentially} proposed differentially private gradient descent method that masks the state by adding Laplace noise. Zhu \textit{et al.}~\cite{zhu2018differentially} extended the above results to time-varying directed networks. However, they only provided the sublinear convergence analysis. To this end, Ding \textit{et al.}~\cite{ding2021differentially } achieved both linear convergence and differential privacy by simultaneously adding noise to states and directions and using constant stepsizes. Notice that, as pointed out in~\cite{ding2021differentially }, it is impossible to achieve differential privacy and accurate convergence simultaneously for the problem~\eqref{P1}. Similar impossibility results can be found in~\cite{huang2024differential} as well. Chen \textit{et al.}~\cite{chen2023differentially} further considered the case of directed graphs. To this end, the authors of~\cite{mo2016privacy,wang2019privacy,he2018preserving,altafini2020system} use correlated noise to avoid the loss of accuracy. Mo and Murray~\cite{mo2016privacy} designed a special time-decaying noise sequence. Wang~\cite{wang2019privacy} proposed a state-decomposition method. The intuition of such approaches is to let the noise sum be zero. However, the level of privacy that can be protected is relatively low due to the correlation between the added noises. The second category of privacy-preserving distributed optimization algorithms involve encryption. For example, Lu and Zhu~\cite{lu2018privacy} proposed a privacy-preserving distributed optimization method using homomorphic encryption. Although encryption-based methods can enable the solutions to converge to the exact optimal, they require a significant amount of computing resources.
	
Most of the aforementioned approaches investigated the privacy-preserving distributed optimization algorithms under the idealized communication network. In practice, it is necessary to consider compressed information due to limited communication bandwidth. Alistarh \textit{et al.}~\cite{alistarh2017qsgd} and Koloskova \textit{et al.}~\cite{koloskova2019decentralized} proposed communication-efficient stochastic gradient descent algorithms by using an unbiased compressor and biased but contractive compressors, respectively. Liao \textit{et al.}~\cite{liao2022compressed} introduced a general class of compressors with bounded relative compression error. They point out that their compressors cover the two types of compressors mentioned above. Kajiyama \textit{et al.}~\cite{kajiyama2020linear} achieved linear convergence by combining the gradient tracking algorithm with a compressor with bounded absolute compression errors, and Xiong \textit{et al.}~\cite{xiong2021quantized} extended the approach in~\cite{kajiyama2020linear} to directed graphs. Additionally, the compressed communication algorithms proposed in~\cite{reisizadeh2019robust,koloskova2019decentralized,taheri2020quantized,yi2022communication} are applicable to nonconvex cost functions.
	 
Due to the advantages of compressed information in saving communication bandwidth, it is natural to consider the marriage between communication compression and privacy preservation. However, there are relatively few related works because of the complex coupling between the compression error and the noise required to achieve privacy. Agarwal \textit{et al.}~\cite{agarwal2018cpsgd} considered a Binomial mechanism and a stochastic quantization in federated learning, which is not suitable for the decentralized scenario with no central servers. Wang and Ba{\c{s}}ar~\cite{wang2022quantization} proposed a differentially private stochastic gradient descent algorithm with compressed communication even for nonconvex cost functions. Both~\cite{agarwal2018cpsgd} and~\cite{wang2022quantization} pointed out that their algorithms can achieve strict $(\epsilon,\delta)$-differential privacy. None of~\cite{agarwal2018cpsgd} and~\cite{wang2022quantization}, however, provided the linear convergence analysis. Besides, their privacy analysis relies on a specific compressor. 
	 
In this paper, we propose compressed, differentially private, distributed, nonconvex optimization algorithms, which preserve differential privacy and achieve 
state-of-the-art linear convergence rates. The main contributions of this work are summarized as follows:
 \begin{enumerate}
	\item For a general class of compressors with bounded relative compression error, we propose a novel nonconvex differentially Private Gradient Tracking algorithm under Compressed communication (PGTC). To guarantee the generality of compressors and preserve privacy, the states will be masked by additional Laplace noises. We show that PGTC converges to a neighborhood of a stationary point with the rate $\mathcal{O}(1/T)$ in general nonconvex settings (Theorem~\ref{theo:convergence0}) and linearly converge to a neighborhood of the global optimum when the global cost function additionally satisfies the Polyak–Łojasiewicz (P--L) condition (Theorem~\ref{theo:convergence}). The size of the neighborhood is determined by the noise added on the gradient. Compared with~\cite{ding2021differentially}, PGTC achieves the same convergence rate even for nonconvex cost functions and bandwidth constrained communication network, and compared with~\cite{agarwal2018cpsgd,wang2022quantization}, we establish the linear convergence rate. 
        \item To further improve communication efficiency, we provide the nonconvex differentially Private Primal-Dual algorithm under Compressed communication  (PPDC), which has similar convergence properties of PGTC (Theorem~\ref{theo:convergence1} and Theorem~\ref{theo:convergence2}). Compared to PGTC, each agent under PPDC only needs to transmit one compressed variable to its neighbors at each time step. Compared with~\cite{yi2022communication}, we further consider the privacy concern. Notice that the noise will be accumulated over time for PGTC and PPDC, which increases the difficulty in analyzing the convergence of the algorithms.
	\item Theoretically, we show that PGTC and PPDC preserve $\epsilon$-differential privacy for the local cost function of each agent even as the time goes to infinity (Proposition~\ref{prop:privacy} and~\ref{propo:privacy1}), but a strict assumption (Assumption~\ref{as:distance1}) is required. Subsequently, we introduce a more general assumption, under which the algorithms can only guarantee privacy for a finite time horizon
 (Theorem~\ref{theo:privacy} and Theorem~\ref{theo:privacy1}). Furthermore, different from~\cite{wang2022quantization,agarwal2018cpsgd}, the privacy under PGTC and PPDC does not rely on some specific compressors but are effective for a general class of compressors.
 \end{enumerate}

The remainder of this paper is organized as follows. In Section~\ref{sec:Problemsetup}, we introduce the preliminaries and formulate the problem. The PGTC algorithm is proposed in Section~\ref{sec:Algorithm}, and its convergence and privacy are then analyzed. Section~\ref{sec:Algorithm1} provides the PPDC algorithm and the corresponding analysis. Some numerical examples are provided in Section~\ref{simulation} to verify the theoretical results. The conclusion and proofs are provided in Section~\ref{conclusion} and Appendix~\ref{app-convergence00}--\ref{app-privacy1}, respectively.
	
\emph{Notations}: $\R$ ($\mathbb{R}_{+}$) is the set of (positive) real numbers. $\mathbb{Z}$ is the set of integers and $\mathbb{N}$ the set of nature numbers. $\mathbb{R}^n$ is the set of $n$ dimensional vectors with real values. The transpose of a matrix $P$ is denoted by $P^\top$, and we use $[P]_{ij}$ to denote the element in its $i$-th row and $j$-th column. The Kronecker product is denoted by~$\otimes$. The $n$-dimensional all-one and all-zero column vectors are denoted by $\mathbf{1}_n$ and $\mathbf{0}_n$, respectively. The $n$-dimensional identity matrix is denoted by $I_n$. We then introduce two stacked vectors: for a vector $\x\in\R^{nd}$, we denote $\bar{x}=\frac{1}{n}(\mathbf{1}_n^\top\otimes I_d)\x$, $\mathbf{\bar{\x}}\triangleq\mathbf{1}_n\otimes\bar{x}$. Further, $\vert\cdot\vert,\Vert\cdot\Vert_1,\Vert\cdot\Vert$ denote the absolute value, $l_1$ norm and $l_2$ norm, respectively. For a matrix $W$ having positive eigenvalues, we use $\bar{\lambda}_W$ and $\underline{\lambda}_W$ to denote its spectral radius and minimum positive eigenvalue respectively. Furthermore, for any square matrix $A$ and vector $x$ with suitable dimension, we denote $\Vert x\Vert_A^2=x^\top Ax$. For a given constant $\theta>0$, $\text{Lap}(\theta)$ is the Laplace distribution with the probability density function $f_L(x,\theta)=\frac{1}{2\theta}e^{-\frac{\vert x\vert}{\theta}}$. For any vector $\xi=[\xi_1,\dots,\xi_d]^\top\in\R^d$, we say that $\xi\sim\text{Lap}_d(\theta)$ if each component $\xi_i\sim\text{Lap}(\theta)$, $i=1,\dots,d$. Furthermore, we use $\E[\cdot]$ and $P[\cdot]$ to denote the expectation of a random variable and the probability of an event, respectively. 

\section{Preliminaries and Problem Formulation} \label{sec:Problemsetup}

\subsection{Standard Assumptions}

For the distributed optimization problem~\eqref{P1}, we consider that each agent $i$ maintains a local estimate $x_{i,k}\in\mathbb{R}^d$ of $x$ at time step $k$ and use $\lf_i(x_{i,k})$ to denote the gradient of $f_i$ at $x_{i,k}$. We make the following assumptions on the local cost functions $f_i$.
\begin{assumption}\label{as:smooth}
	Each local cost function $f_i$ is $L_f$-smooth, for some $L_f>0$, namely for any $x,y\in\mathbb{R}^d$,
	\begin{align}\label{eqn:smooth}
	\left\Vert \lf_i(x)-\lf_i(y)\right\Vert\leq L_f\left\Vert x-y\right\Vert.
	\end{align}
\end{assumption}
From~\eqref{eqn:smooth}, we have
\begin{align}\label{eqn:smooth1}
\vert f_i(y)-f_i(x)-(y-x)^\top\lf_i(x)\vert\leq\frac{L_f}{2}\left\Vert y-x\right\Vert^2.
\end{align}
\begin{assumption}\label{as:finite}
Let $f^*$ be the minimum function value of the problem~\eqref{P1}. We assume $f^*>-\infty$.
\end{assumption}
\begin{assumption}\label{as:PLcondition}
(Polyak–Łojasiewicz (P--L) condition~\cite{yi2022communication}) There exists a constant $\nu>0$ such that for any $x\in\mathbb{R}^d$,
\begin{align}
\frac{1}{2}\left\Vert \lf(x)\right\Vert^2\geq \nu(f(x)-f^*).
\end{align}
\end{assumption}
\begin{remark}
Assumptions~\ref{as:smooth}--\ref{as:finite} are standard in distributed nonconvex optimization, see e.g., \cite{shi2015extra,yang2019survey,yi2022communication}. Furthermore, as point out in~\cite{yi2022communication}, the P--L condition is a weaker assumption than strong convexity and ensures that each stationary point of problem~\eqref{P1} is a global minimizer.
\end{remark}

\subsection{Basics of Graph Theory}
	
The exchange of information between the $n$ agents is captured by an undirected graph $\mathcal{G}(\mathcal{V}, \mathcal{E})$ of $n$ nodes, where $\mathcal{V}=\{1, 2, \ldots, n\}$ is the set of the agents' indices and $\mathcal{E} \subseteq \mathcal{V} \times \mathcal{V}$ is the set of edges. The edge $(i,j)\in\mathcal{E}$ if and only if agents~$i$ and~$j$ can communicate with each other. Let $W=[w_{ij}]_{n\times n}\in\mathbb{R}^{n\times n}$ be the positively weighted adjacency matrix of $\mathcal{G}$, namely $w_{ij}>0$ if $(i,j)\in\mathcal{E}$, and $w_{ij}=0$, otherwise. Note that $w_{ii}=0$ due to the self edge $(i,i)\notin\mathcal{E}$. We use $\mathcal{N}_i=\{j\in\mathcal{V}|~(i,j)\in\mathcal{E}\}$ to denote the neighbor set of agent $i$ and use $D=\text{diag}[d_1,d_2,\cdots,d_n]$ to denote the degree matrix, where $d_i=\sum_{j}^n w_{ij},~\forall i\in\mathcal{V}$. The Laplacian matrix of graph $\mathcal{G}$ is denoted by $L=D-W$.
	
\begin{assumption}\label{as:strongconnected}
	The undirected graph $\mathcal{G}(\mathcal{V}, \mathcal{E})$ is connected and $W$ is a doubly stochastic matrix, i.e., $\mathbf{1}^\top W=\mathbf{1}^\top$ and $W\mathbf{1}=\mathbf{1}$.
\end{assumption}

\subsection{Compression Method}
To improve the communication efficiency, we consider the situation where agents compress the information before sending it. More specifically, for any $x\in\R^d$, we consider a general class of stochastic compressors $C(x,\varrho)$ and use $f_c(x,\varrho)$ to denote the corresponding probability density functions, where $\varrho$ is a random perturbation variable. Furthermore, the compressors $C(x,\varrho)$ can be simplified to $C(x)$ when the distribution of $\varrho$ is given. We then introduce the following assumption.

\begin{assumption}\label{as:compressor}
For some $\varphi\in(0,1]$ and $r>0$, the stochastic compressor $C(\cdot):\mathbb{R}^d\mapsto\mathbb{R}^d$ satisfies 
\begin{align*}
    \mathbb{E}_C\left[
        \left\Vert \frac{C(x)}{r}-x\right\Vert^2\right]\leq(1-\varphi)\left\Vert x\right\Vert^2, \forall x\in\mathbb{R}^d,\addtag\label{eq:propertyofcompressors}
\end{align*}
where $\mathbb{E}_C$ denotes the expectation over the internal randomness of the stochastic compression operator $C$. 
\end{assumption}
From~\eqref{eq:propertyofcompressors} and the Cauchy-Schwarz inequality, one obtains that
\begin{align}\label{eq:propertyofcompressors1}
    \mathbb{E}_C\left[
    \left\Vert C(x)-x\right\Vert^2\right]\leq r_0\left\Vert x\right\Vert^2, \forall x\in\mathbb{R}^d.
\end{align}
where $r_0=2r^2(1-\varphi)+2(1-r)^2$.
\begin{remark}
Compressors under Assumption~\ref{as:compressor} are general. As pointed out in~\cite{liao2022compressed}, the compressors satisfying the Assumption~\ref{as:compressor} cover a class of unbiased compressors~\cite{alistarh2017qsgd,liu2021linear} and biased but contractive compressors~\cite{reisizadeh2019exact,koloskova2019decentralized,taheri2020quantized}. Noting that, the compressors satisfying Assumption~\ref{as:compressor} also cover the compressors used in~\cite{wang2022quantization}. Furthermore, it is important to note that $\mathbb{E}_C[
\Vert C(x)-x\Vert^2]=0$ if $\varphi=1,~r=1$, which means that the uncompressed case is also included in Assumption~\ref{as:compressor}. In other words, using compressors satisfying Assumption~\ref{as:compressor} alone is not sufficient to ensure privacy,, and we need to introduce additional stochasticity.
\end{remark}

\subsection{Differential Privacy}

To evaluate the privacy performance, we adopt the notion of $\epsilon$-differential privacy for the distributed optimization, which has recently been studied in~\cite{huang2015differentially,ding2021differentially}. Specifically, we introduce the following definitions.

\begin{definition}\label{def:adjacent}
\textbf{(Adjacency~\cite{chen2023differentially})} Two function sets $\mathcal{S}^{(1)}=\{f_i^{(1)}\}_{i=1}^n$ and $\mathcal{S}^{(2)}=\{f_i^{(2)}\}_{i=1}^n$ are said to be adjacent if there exists some $i_0\in\{1,2,\dots,n\}$ such that
\begin{align*}
    f_i^{(1)}=f_i^{(2)}~\forall i\neq i_0,~\text{and}~f_{i_0}^{(1)}\neq f_{i_0}^{(2)}.
\end{align*}
\end{definition}

Given a cost function set $\mathcal{S}$, we denote the randomized mechanism as a mapping $\mathbb{M}({\mathcal{S}},x_0):x_0\mapsto\mathcal{H}$, where $x_0$ and $\mathcal{H}$ are the initial state and observation, respectively.
\begin{definition}\label{def:differentialprivacy}
\textbf{(Differential privacy~\cite{chen2023differentially})} Given $\epsilon>0$ and a randomized mechanism $\mathbb{M}$, for any two adjacent function sets $\mathcal{S}^{(1)}$ and $\mathcal{S}^{(2)}$, any initial state $x_0$ and any observation $\mathcal{H}\subseteq\text{Range}(\mathbb{M})$, the randomized mechanism $\mathbb{M}$ keeps $\epsilon$-differential privacy if
\begin{align}\label{eq:dp}
P\{\mathbb{M}(\mathcal{S}^{(1)},x_0)\in\mathcal{H}\}\leq e^\epsilon P\{\mathbb{M}(\mathcal{S}^{(2)},x_0)\in\mathcal{H}\},
\end{align}
where Range($\mathbb{M}$) denotes the output domain of $\mathbb{M}$.
\end{definition}

Definition~\ref{def:differentialprivacy} shows that the randomized mechanism $\mathbb{M}$ is $\epsilon$-differential private if for any pair of adjacent function sets, the probability density functions of their observations are similar. Intuitively, it is difficult for an adversary to distinguish between two adjacent function sets merely by observations if the corresponding mechanism $\mathbb{M}$ is $\epsilon$-differential private. It is worth noting, as pointed out in~\cite{ding2021differentially,huang2024differential}, that achieving both accurate convergence and strict $\epsilon$-differential privacy (see Definition~\ref{def:differentialprivacy}) simultaneously for Problem~\eqref{P1} is impossible. Intuitively, privacy is guaranteed when the perturbation (noise) is large enough, and more details can be found in~\cite[Proposition~1]{ding2021differentially} and~\cite[Theorem~1]{huang2024differential}. Therefore, in this paper, we are more concerned about the trade-off between privacy and accuracy.

\section{Distributed Gradient tracking Algorithm with Compressed Communication}\label{sec:Algorithm}
In this section, we provide the nonconvex differentially Private Gradient Tracking algorithm under Compressed communication  (PGTC), which is shown in Algorithm~\ref{Al:PGTC}.

\subsection{Algorithm Description}
The proposed PGTC is inspired by the DiaDSP Algorithm~\cite{ding2021differentially}. In this paper, we additionally consider the compressed information and the nonconvex cost functions. We first assume that each agent $i\in\mathcal{V}$ maintains an estimate $x_{i,k}$ and an auxiliary variable $y_{i,k}$ for tracking the global gradient. To guarantee differential privacy, each agent $i$ broadcasts the noisy $x_{i,k}^a$ and $y_{i,k}^a$ to its neighbors $\mathcal{N}_i$ per step, where
\begin{align*}
&x_{i,k}^a=x_{i,k}+\xi_{x_i,k},\addtag\label{eq:addnoisex}\\
&y_{i,k}^a=y_{i,k}+\xi_{y_i,k},\addtag\label{eq:addnoisey}
\end{align*}
and $\xi_{x_i,k}$ and $\xi_{y_i,k}$ are Laplace noises. Similar to the DiaDSP Algorithm~\cite{ding2021differentially}, we set $\xi_{x_i,k}\sim\text{Lap}_d(s_{\xi_{x_i}}q_i^k)$ and $\xi_{y_i,k}\sim\text{Lap}_d(s_{\xi_{y_i}}q_i^k)$, $\forall i\in\mathcal{V}$, where $s_{\xi_{x_i}}>0,~s_{\xi_{y_i}}>0$, and $0<q_i<1$. After the information exchange, agent $i$ performs the following updates:
\begin{align}
	&x_{i,k+1}=\sum_{j=1}^n w_{ij}x_{j,k}^a-\eta y_{i,k},\\
	&y_{i,k+1}=\sum_{j=1}^n w_{ij}y_{j,k}^a+\lf_i(x_{i,k+1})-\lf_i(x_{i,k}),
\end{align}
where the stepsize $\eta$ is a constant and the initial value $y_{i,0}=\lf_i(x_{i,0}),~\forall i\in\mathcal{V}$. To improve the communication efficiency, we introduce a general class of compressors $C(\cdot)$ and use the revised compressed variable $\hat{x}_{i,k}$, $\hat{y}_{i,k}$ to replace $x_{i,k}^a$, $y_{i,k}^a$, respectively. Noting that if the compressed variable $C(x_{i,k}^a)$ and $C(y_{i,k}^a)$ are directly used here, the compression error will be accumulate and affect the convergence. Then, we design the updates of agent $i\in\mathcal{V}$ as follows:
\begin{align*}
	&~x_{i,k+1}=x_{i,k}^a+\gamma\sum_{j=1}^n w_{ij}(\hat{x}_{j,k}-\hat{x}_{i,k})-\eta y_{i,k},\addtag\label{iterationx}\\
	&~y_{i,k+1}=y_{i,k}^a+\gamma\sum_{j=1}^n w_{ij}(\hat{y}_{j,k}-\hat{y}_{i,k})+\lf_i(x_{i,k+1})-\lf_i(x_{i,k}),\addtag\label{iterationy}
\end{align*}
where
\begin{align*}
	&\hat{x}_{j,k}=x_{j,k}^c+C(x_{j,k}^a-x_{j,k}^c),\addtag\label{citerationx}\\
	&\hat{y}_{j,k}=y_{j,k}^c+C(y_{j,k}^a-y_{j,k}^c),\addtag\label{citerationy}\\
	&x_{j,k+1}^c=(1-\alpha_x)x_{j,k}^c+\alpha_x\hat{x}_{j,k},\addtag\label{citerationxc}\\
	&y_{j,k+1}^c=(1-\alpha_y)y_{j,k}^c+\alpha_y\hat{y}_{j,k},\addtag\label{citerationyc}
\end{align*}
with $\gamma$, $\alpha_x$, and $\alpha_y$ being some positive parameters. We assume that $x_{i,0}^c=\mathbf{0}$ and $y_{i,0}^c=\mathbf{0}$,~$\forall i\in\mathcal{V}$. Let $W_\gamma\triangleq((1-\gamma)I_n+\gamma W)\otimes I_d$, and then \eqref{iterationx} and \eqref{iterationy} can be rewritten into the following compact form
\begin{align*}
	&\mathbf{x}_{k+1}=W_\gamma(\x_k+\nxk)+\gamma(W-I_n)\otimes I_d(\hx_k-\x_k-\nxk)-\eta\y_k,\addtag\label{iterationx1}\\
	&\mathbf{y}_{k+1}=W_\gamma(\y_k+\nyk)+\gamma(W-I_n)\otimes I_d(\hy_k-\y_k-\nyk)+\f(\x_{k+1})-\f(\x_k),\addtag\label{iterationy1}
\end{align*}
where $\x_k\triangleq[x_{1,k}^\top,\dots,x_{n,k}\top]^\top\in\R^{nd}, \y_k\triangleq[y_{1,k}^\top,\dots,$ $y_{n,k}^\top]^\top\in\R^{nd}, \hx_k\triangleq\left[\hat{x}_{1,k}^\top,\dots,\hat{x}_{n,k}^\top\right]^\top\in\R^{nd},\hy_k$ $\triangleq[\hat{y}_{1,k}^\top,\dots,\hat{y}_{n,k}^\top]^\top\in\R^{nd},~\f(\mathbf{x}_k)\triangleq[\lf_1(x_{1,k})^\top,\dots,$ $\lf_n(x_{n,k})^\top]^\top\in\R^{nd}, \nxk\triangleq[\xi_{x_{1},k}^\top,\dots,\xi_{x_{n},k}^\top]^\top\in\R^{nd},\nyk\triangleq[\xi_{y_{1},k}^\top,\dots, \xi_{y_{n},k}^\top]^\top\in\R^{nd}$.
\begin{algorithm}[]
	\caption{PGTC Algorithm }
	\label{Al:PGTC}
	\begin{algorithmic}[1]
		\STATE \textbf{Input:} Stopping time $K$, adjacency matrix $W$, and positive parameters $\eta$, $\gamma$, $\alpha_x$, $\alpha_y$, $s_{\xi_{x_i}}$, $s_{\xi_{y_i}}$, $q_i$, $\forall i\in\mathcal{V}$.
		\STATE \textbf{Initialization:} Each ~$i\in\mathcal{V}$ chooses arbitrarily $x_{i,0}\in\mathbb{R}^d$, $x^c_{i,0}=\bf{0}$, $y^c_{i,0}=\bf{0}$, and computes $y_{i,0}=\lf_i(x_{i,0})$.
		\FOR{$k=0,1,\dots,K-1$}
		\FOR {for $i\in\mathcal{V}$ in parallel} 
		\STATE Generate Laplace noises $\xi_{x_i,k}\sim\text{Lap}_d(s_{\xi_{x_i}}q_i^k)$ and $\xi_{y_i,k}\sim\text{Lap}_d(s_{\xi_{y_i}}q_i^k)$.
		\STATE Obtain $x_{i,k}^a$ and $y_{i,k}^a$ from~\eqref{eq:addnoisex} and~\eqref{eq:addnoisey}, respectively.
		\STATE Compute $C(x_{i,k}^a-x_{i,k}^c)$ and $C(y_{i,k}^a-y_{i,k}^c)$, then broadcast them to its neighbors $\mathcal{N}_i$.
		\STATE Receive $C(x_{j,k}^a-x_{j,k}^c)$, and $C(y_{j,k}^a-y_{j,k}^c)$ from $j\in\mathcal{N}_i$.
		\STATE Update $\hat{x}_{i,k}$, $\hat{y}_{i,k}$, $x_{i,k+1}^c$, and $y_{i,k+1}^c$, $\forall j\in\mathcal{N}_i\cup\{i\}$, from~\eqref{citerationx}--\eqref{citerationyc}, respectively.
		\STATE Update $x_{i,k+1}$ and $y_{i,k+1}$ from~\eqref{iterationx} and~\eqref{iterationy}, respectively.

		\ENDFOR
		\ENDFOR
		\STATE \textbf{Output:} \{$x_{i,k}$\}.
	\end{algorithmic}
\end{algorithm}

\subsection{Convergence Analysis of PGTC}
In this section, we will analyze convergence of PGTC under the compressors satisfying Assumption~\ref{as:compressor}. 

Let $\Theta_k\triangleq[\oxk, \oyk, \osxk, \osyk]^\top$, where $\oxk=\Vert \x_k-\bx_k\Vert^2$, $\oyk=\Vert \y_k-\by_k\Vert^2$, $\osxk=\Vert \sxk\Vert^2$, and $\osyk=\Vert \syk\Vert^2$, with $\sxk\triangleq[(x_{1,k}^c-x_{1,k}^a)^\top,\dots,(x_{n,k}^c-x_{n,k}^a)^\top]^\top\in\R^{nd},~\syk\triangleq[(y_{1,k}^c-y_{1,k}^a)^\top,\dots,(y_{n,k}^c-y_{n,k}^a)^\top]^\top\in\R^{nd}$. The following lemma constructs a set of linear inequalities that is related to $\Theta_k$.

\begin{lemma}\label{lemma:linearinequalities}
	Suppose Assumptions~\ref{as:smooth}--\ref{as:finite} and \ref{as:strongconnected}--\ref{as:compressor} hold. Under Algorithm~\ref{Al:PGTC}, if $\alpha_x,\alpha_y\in(0,\frac{1}{r})$, we have the following linear inequalities:
	\begin{align}\label{eq:linearinequalities}
		\mathbb{E}[\Theta_{k+1}]\preceq G\mathbb{E}[\Theta_k]+\vartheta_1\E[\left\Vert\by_k\right\Vert^2]+\vartheta_2 \bar{q}^{2k}\bar{s_\xi}^2,
	\end{align}
where $\bar{q}=\max_i\{q_i\}$, $\bar{s_\xi}=\max_i\{s_{\xi_{x_i}},s_{\xi_{y_i}}\}$, the notation $\preceq$ means element-wise less than or equal to, the matrix $G\in\mathbb{R}^{4\times 4}$ and vectors $\vartheta_1,\vartheta_2\in\R^{4}$ are given in Appendix~\ref{app-convergence00}.
\end{lemma}
\begin{proof}
	See Appendix~\ref{app-convergence00}.
\end{proof}
To analyze the convergence of PGTC, we choose the following Lyapunov function
\begin{align}\label{eqn:lyapunov}
\begin{aligned}
V_k&=\E[f(\bar{x}_k)]-f^*+\frac{\zeta_1L}{n}\E[\oxk]+\frac{\zeta_2}{nL}\E[\oyk]+\frac{\zeta_3L}{n}\E[\osxk]+\frac{\zeta_4}{nL}\E[\osyk],\\
&=\E[f(\bar{x}_k)]-f^*+\mathbf{s}^\top\E[\Theta_k],
\end{aligned}
\end{align} 
where $\bar{x}_k=\frac{\mathbf{1}^\top}{n}\x_k$, $\zeta_1$--$\zeta_4$ are some positive constants that will be given later, 
\begin{align*}
	\mathbf{s}=\left[\frac{\zeta_1L}{n}~\frac{\zeta_2\rho^2}{nL}~\frac{\zeta_3L}{n}~\frac{\zeta_4\rho^2}{nL}\right]
\end{align*}
with $\rho=1-\bar{\lambda}_{W-\frac{\mathbf{1}\mathbf{1}^\top}{n}}$. Since $f^*$ is the minimum function value, we know that the Lyapunov function $V_k$ is well defined.

We now show the convergence results of PGTC.
\begin{theorem}\label{theo:convergence0}
Suppose Assumptions~\ref{as:smooth}--\ref{as:finite} and \ref{as:strongconnected}--\ref{as:compressor} hold. Under Algorithm~\ref{Al:PGTC}, assume $\alpha_x,\alpha_y\in(0,\frac{1}{r})$ and let $\gamma=\zeta_\gamma\rho\varphi_1,~\eta=\zeta_\eta\gamma\rho^2/L_f$, where $0<\zeta_\gamma=\zeta_\eta\leq\bar{\zeta}_1$. For any $T\in\mathbb{N}$, it holds that 
\begin{align*}
 &\frac{1}{T}\sum_{k=0}^{T}\big(\mathbb{E}[\Vert\lf(\bar{x}_k)\Vert^2]+\mathbb{E}[\Vert \x_k-\bx_k\Vert^2]\big)\leq\frac{\bar{\kappa}_1 M_1}{T}+\frac{4}{n}\mathbb{E}\left[\left\Vert\sum_{t=0}^\infty\xi_{y,t}\right\Vert^2\right],\addtag\label{eq:theo1}
\end{align*}
where $\bar{\kappa}_1$, $M_1$, $\bar{\zeta}_1$ and $\zeta_1$--$\zeta_4$ are constants given in Appendix~\ref{app-convergence0}.
\end{theorem}
\begin{proof}
	See Appendix~\ref{app-convergence0}.
\end{proof}

Due to the fact that $\xi_{y_i,k},~\forall i\in\mathcal{V},~\forall k\in\mathbb{N}$ are 
independent of each other, $\E\left\Vert\sum_{t=0}^\infty\xi_{y,t}\right\Vert^2$ can be rewritten as $\sum_{t=0}^\infty\E\Vert\xi_{y,t}\Vert^2$. Then we have the following result.
\begin{corollary}\label{coro0}
    Under the same assumptions and parameters in Theorem~\ref{theo:convergence0}. It holds that
\begin{align*}
 \frac{1}{T}\sum_{k=0}^{T}\big(\mathbb{E}[\Vert\lf(\bar{x}_k)\Vert^2]+\mathbb{E}[\Vert \x_k-\bx_k\Vert^2]\big)\leq\frac{\bar{\kappa}_1 M_1}{T}+\frac{8d\bar{s}_{\xi}^2}{1-\bar{q}^2},
\end{align*}
where $d$ is the dimension of the state vector $x$ and $\bar{s}_{\xi}$ and $\bar{q}$ are defined in Lemma~\ref{lemma:linearinequalities}.
\end{corollary}
\begin{remark}
    Theorem~\ref{theo:convergence0} and Corollary~\ref{coro0} show that PGTC converges to a neighborhood of a stationary point with the rate $\mathcal{O}(1/T)$ for general nonconvex cost functions. The same convergence rate was achieved by algorithms proposed in~\cite{yi2022communication,zhao2022beer} under the same assumptions and cost function. However, they do not consider the privacy concern. Furthermore, since the noise added to the gradient tracking is accumulative, the convergence is affected by $\sum_{t=0}^\infty\xi_{y,t}$. More details will be given later. Furthermore, Theorem~\ref{theo:convergence0} does not require $\gamma$ and $\eta$ to be some fixed constants due to the fact that $0<\zeta_\gamma=\zeta_\eta\leq\bar{\zeta}_1$. It is only necessary to select a $\zeta_\gamma=\zeta_\eta$ that satisfies the above condition in the implement.
\end{remark}
Then we provide the linear convergence analysis with the P--L condition.
\begin{theorem}\label{theo:convergence}
Suppose Assumptions~\ref{as:smooth}--\ref{as:compressor} hold. Under Algorithm~\ref{Al:PGTC}, if $\alpha_x,\alpha_y\in(0,\frac{1}{r})$ and let $\gamma=\zeta_\gamma\rho\varphi_1,~\eta=\zeta_\eta\gamma\rho^2/L_f$, where $0<\zeta_\gamma=\zeta_\eta\leq\bar{\zeta}_2$. It holds that
\begin{align*}
	\E[V_{k+1}]\leq (1-\frac{\eta\nu}{2})\E[V_k]+b_1 \bar{q}^{2k}\bar{s_\xi}^2+\frac{\eta}{n}\E\left\Vert\sum_{t=0}^k\xi_{y,t}\right\Vert^2,
\end{align*}
where $b_1$ $M_1$, $\bar{\zeta}_2$ and $\zeta_1$--$\zeta_4$ are constants given in Appendices~\ref{app-convergence0} and~\ref{app-convergence}.
\end{theorem}
\begin{proof}
	See Appendix~\ref{app-convergence}.
\end{proof}
Similar to the way we obtained Corollary~\ref{coro0}, we have

\begin{corollary}\label{coro}
    Under the same assumptions and parameters in Theorem~\ref{theo:convergence}. It holds that
\begin{align*}
	\E[V_{k+1}]\leq (1-\bar{\kappa}_2)^{k+1}(\E[V_0]+\frac{b_1\bar{s_\xi}^2}{1-\bar{\kappa}_2-\bar{q}^2})+\frac{4d\bar{s}_{\xi}^2}{(1-\bar{q}^2)\nu},
\end{align*}
where~$0<\bar{\kappa}_2<\min\{\frac{\eta\nu}{2},1-\bar{q}^2\}$.
\end{corollary}

From~\eqref{eqn:lyapunov}, it can be observed from Theorem~\ref{theo:convergence} and Corollary~\ref{coro} that  $\E[f(\bar{x}_k)]-f^*+\frac{\zeta_1L}{n}\E[\oxk]=\mathcal{O}((1-\bar{\kappa}_2)^k)+\mathcal{O}(1)$, which means PGTC linearly converges to a neighborhood of the global optimum under P--L condition. However, the size of the neighborhood is determined by the noise accumulated over time on gradients. This is because the noise on gradients accumulates over the iterations. More specifically,  from~\eqref{iterationy1}, we can see that $\frac{1}{n}(\mathbf{1}_n^\top\otimes I_d)\y_k=\frac{1}{n}(\mathbf{1}_n^\top\otimes I_d)(\f(\x_k)+\sum_{t=0}^{k-1}\xi_{y,t})$. As pointed out in~\cite{zhu2019deep}, the differential privacy is achieved only when the noise variance is large enough to affect accuracy. Although there are technical means to avoid the accumulation of noise, e.g., \cite{xiong2021quantized}, it is necessary to reserve the the accumulated noise term for privacy protection. A similar result was also established by DiaDSP proposed in~\cite{ding2021differentially}. However, DiaDSP only works for ideal communication network and the authors did not provide the analysis for nonconvex cost functions. Moreover, DiaDSP demonstrates that the convergence point $x^\infty\in\R^d$ satisfies the following property when the cost functions are strongly convex and smooth:
\begin{align}\label{eq:xinfty}
	\sum_{i=1}^n\lf_i(x^\infty)=-\sum_{i=1}^n\sum_{k=0}^\infty \eta_{y_i,k}.
\end{align}
However, it is important to note that the analysis in Theorem~\ref{theo:convergence} does not yield the same result due to the limitations of the P--L condition. As shown in Assumption~\ref{as:PLcondition}, the P--L condition only establishes the gradient relationship between any point and the optimal point, whereas the strongly convex condition used in~\cite{ding2021differentially} establishes a similar relationship between any two points. Nevertheless, under strongly convex cost functions, the convergence point of PGTC coincides with DiaDSP~\cite{ding2021differentially}, suggesting that the proposed PGTC achieves a comparable level of accuracy to the algorithm with idealized communication, assuming the same cost functions. For additional details, please refer to the our previous work~\cite{xie2023compressed}.

\subsection{$\epsilon$-Differential privacy}
In this section, we show that the differential privacy of all cost functions can be preserved under PGTC. 

We use $\mathcal{H}_k$ to denote the information transmitted between agents at time step $k$, i.e., $\mathcal{H}_k=\{C(x_{i,k}^a-x_{i,k}^c),C(y_{i,k}^a-y_{i,k}^c)|~\forall i\in\mathcal{V}\}$. Without loss of generality, we assume the adversary aims to infer the cost function of agent $i_0$. Consider any two 
adjacent function sets $\mathcal{S}^{(1)}$ and $\mathcal{S}^{(2)}$, and only the cost function $f_{i_0}$ is different between the two sets, i.e., $f_{i_0}^{(1)}\neq f_{i_0}^{(2)}$ and $f_{i}^{(1)}= f_{i}^{(2)},~\forall i\neq i_0$. Before provide the privacy result, we first introduce the following constrained assumption~\cite{ding2021differentially}.

\begin{assumption}~\cite{ding2021differentially}\label{as:distance1}
    For any $x_1,~x_2\in\R^d$, we have
    \begin{align*}
    \lf_{i_0}^{(1)}(x_1)-\lf_{i_0}^{(1)}(x_2)=\lf_{i_0}^{(2)}(x_1)-\lf_{i_0}^{(2)}(x_2).
    \end{align*}
\end{assumption}
\begin{proposition}\label{prop:privacy}
	Suppose Assumptions~\ref{as:smooth}--\ref{as:distance1} hold. PGTC preserves the $\epsilon_{i_0}$-differential privacy for any agent $i_0$'s cost function if $\eta<\frac{1}{2L_f}$ and $q_{i_0}\in(\frac{\eta L_f+\sqrt{\eta^2L_f^2+4\eta L_f}}{2},1)$, where $\epsilon_{i_0}$ is given by
\begin{align}\label{eqn:proprivacy}
	\epsilon_{i_0}=\frac{\tau_{i_0} q_{i_0}^2\delta}{q_{i_0}^2-\eta L_f-q_{i_0}\eta L_f},~~~~\forall i_0\in\mathcal{V}
\end{align}
with $\tau_{i_0}=\frac{\eta}{s_{\xi_{x_{i_0}}}}+\frac{1}{s_{\xi_{y_{i_0}}}}$.
\end{proposition}
\begin{proof}
	The proof can be obtained in the same way as the proof of Theorem~2 in\cite{xie2023compressed}.
\end{proof}

Proposition~\ref{prop:privacy} shows that PGTC ensures $\epsilon$-differential privacy even as time goes to infinity but needs a strict assumption (Assumption~\ref{as:distance1}). From~\eqref{eqn:proprivacy}, the privacy budget $
\epsilon_{i_0}$ can be arbitrarily chosen by setting specific values for parameters $s_{\xi_{x_{i_0}}}$ and $s_{\xi_{y_{i_0}}}$. However, higher levels of privacy also imply worse convergence accuracy. To relax the restrictions of Assumption~\ref{as:distance1}, we consider the following more general assumption.

\begin{assumption}\label{as:distance}
The gradient of all local cost functions are bounded, i.e., there exists a positive constant~$\mathbb{M}$ such that~$\Vert\nabla f_i(x)\Vert\leq M$, $\forall~i\in\mathcal{N},~x\in\R^d$.
\end{assumption}
\begin{remark}
    Assumption~\ref{as:distance} is very common used in privacy-preserving distributed optimization problem, e.g.,~\cite{huang2015differentially,zhu2018differentially,chen2023differentially}. It is useful for analyzing differential privacy because it controls the differences between the gradients of adjacent cost functions. This is also the reason why it can be used to relax Assumption~\ref{as:distance1}.
\end{remark}

\begin{theorem}\label{theo:privacy}
	Suppose Assumption~\ref{as:distance} holds, given a finite number of iterations $K$, PGTC preserves the $\epsilon_{i_0}$-differential privacy for any agent $i_0$'s cost function if the parameters satisfy
\begin{align*}	4\sqrt{d}M\sum_{k=0}^K\left(\frac{\sqrt{\eta}}{s_{\xi_{x_{i_0}}}q_{i_0}^k}+\frac{1}{s_{\xi_{y_{i_0}}}q_{i_0}^k}\right)\leq \epsilon_{i_0}.\addtag\label{eq:theo3}
\end{align*}
\end{theorem}
\begin{proof}
	See Appendix~\ref{app-privacy}.
\end{proof}

\begin{remark}
We would like to highlight that, unlike the privacy of the methods in~\cite{wang2022quantization,agarwal2018cpsgd}, which only work for specific compressors, PGTC is effective for a class of compressors. Furthermore, as previously discussed, compared to~\cite{ding2021differentially}, Theorem~\ref{theo:privacy} establishes weaker privacy but only requires mild assumptions. More specifically, the condition~\eqref{eq:theo3} is difficult to be satisfied when the iterations $K$ tends to infinity. In other words, under a weaker assumption (Assumption~\ref{as:distance}), the PGTC can only preserve the privacy within the interval $[0,K]$ for some finite iterations $K$.
\end{remark}
\subsection{Proof Sketch}
We then provide a proof sketch of Theorems~\ref{theo:convergence0} and~\ref{theo:convergence}. Unlike the ideal communication algorithm DiaDSP~\cite{ding2021differentially}, to establish the convergence of PGTC, we need to track the consensus errors and the extra compressed errors of the state $\mathbf{x}_k$ and the estimated gradient $\mathbf{y}_k$ using $\Theta_k$. To estimate those errors, we construct a set of linear inequalities, which are stated in Lemma~\ref{lemma:linearinequalities}. Recalling inequality~\eqref{eq:linearinequalities}, due to the accumulation of noise in the gradient tracking term, i.e.
\begin{align}\label{eq:gradienttracking}
   \by_k=\bbf(\x_k)+\sum_{t=0}^{k-1}\bar{\xi}_{y,t}.
\end{align}
We know that $\mathbb{E}[\left\Vert\bar{\mathbf{y}}_k\right\Vert^2]$ cannot decrease to zero. To estimate the optimization errors and distinguish the redundant parts $\sum_{t=0}^\infty\xi_{y,t}$, we combine the form $f(\bar{x}_k)-f^*$ and $L_f$-smooth, and use the Lyapunov function~\eqref{eqn:lyapunov} to achieve the convergence results in Theorems~\ref{theo:convergence0} and~\ref{theo:convergence}. Notice that since there are no assumptions of strongly convex or convex cost functions, we do not use the inequality property associated with convexity in the proofs. Furthermore, compared with~\cite{agarwal2018cpsgd,wang2022quantization}, to analyze linear convergence, we introduce the gradient tracking method and use constant stepsize in PGTC. However, the gradient tracking method leads to noise accumulation, as we stated before. Finally, the proof of Theorems~\ref{theo:convergence0} and~\ref{theo:convergence} are inspired by the proof of Theorems 4.3 and 4.4 in~\cite{zhao2022beer}. However, since this paper considers more general compressors and privacy, we need to analyze the different compression errors and noise for PGTC.

\section{Distributed Primal-Dual Algorithm with Compressed Communication}\label{sec:Algorithm1}
It can be seen from the PGTC algorithm that each agent should transmit two compressed variables in each iteration. To further improve communication efficiency, in this section, we provide the nonconvex differentially Private Primal-Dual algorithm under Compressed communication (PPDC), which is shown in Algorithm~\ref{Al:PPDC}. Compared with PGTC, each agent under PPDC only needs to transmit one compressed variable to its neighbors in each iteration. This means that PPDC consumes fewer communication resources. 
\subsection{Algorithm Description}
To solve the distributed nonconvex optimization problem~\ref{P1}, Yi~\textit{et~al.}~\cite{yi2021linear} proposed the following distributed primal-dual algorithm
\begin{align*}
&~x_{i,k+1}=x_{i,k}-\eta(\gamma\sum_{j=1}^n L_{ij}x_{j,k}+\omega  v_{i,k}+\lf_i(x_{i,k})),\addtag\label{iterationxp}\\
&~v_{i,k+1}=v_{i,k}+\eta\omega  \sum_{j=1}^n L_{ij}x_{j,k},\addtag\label{iterationv}
\end{align*}
where $\gamma ,~\omega $ are positive parameters, $\eta$ is stepsize, $L_{ij}$ is the $i$-th row and $j$-th column element of the Laplacian matrix $L$ and $v_{i,k}$ is the auxiliary variable of agent $i$. Similar to PGTC, to enable differential privacy, we propose the following algorithm
\begin{align*}
&~x_{i,k+1}=x_{i,k}+\xi_{x_i,k}-\eta(\gamma \sum_{j=1}^n L_{ij}(x_{j,k}+\xi_{x_j,k})+\omega  v_{i,k}+\lf_i(x_{i,k})),\addtag\label{iterationxp1}\\
&~v_{i,k+1}=v_{i,k}+\xi_{v_i,k}+\eta\omega\sum_{j=1}^n L_{ij}(x_{j,k}+\xi_{x_j,k}),\addtag\label{iterationv1}
\end{align*}
with $\xi_{x_i,k}$ and $\xi_{v_i,k}$ are Laplace noises. Similar to the PGTC, we set $\xi_{x_i,k}\sim\text{Lap}_d(s_{\xi_{x_i}}q_i^k)$ and $\xi_{v_i,k}\sim\text{Lap}_d(s_{\xi_{v_i}}q_i^k)$, $\forall i\in\mathcal{V}$, where $s_{\xi_{x_i}}>0,~s_{\xi_{v_i}}>0$, and $0<q_i<1$. Although $v_{i,k}$ is not transmitted directly to the neighbors of agent $i$, both noises $\xi_{x_i,k}$ and $\xi_{v_i,k}$ are needed to enable differential privacy. Specifically, according to Definition~\ref{def:differentialprivacy}, differential privacy requires that the observations under any two adjacent function sets are the same with some positive probability. This means that we need to mask the gradient changes with noises. Similarly, to ensure that the noise is sufficient to protect privacy, we add noise $\xi_{v_i,k}$ to the dual variable $v_{i,k}$ even if it does not need to be transmitted. To improve communication efficiency, we use the compressed variable $\hat{x}_{i,k}$ to replace $x_{i,k}^a$. The updates for agent $i\in\mathcal{V}$ can be designed as follows:
\begin{align*}
&~x_{i,k+1}=x_{i,k}+\xi_{x_i,k}-\eta(\gamma \sum_{j=1}^n L_{ij}\hat{x}_{j,k}+\omega  v_{i,k}+\lf_i(x_{i,k})),\addtag\label{eq:iterationxp2}\\
&~v_{i,k+1}=v_{i,k}+\xi_{v_i,k}+\eta\omega \sum_{j=1}^n L_{ij}\hat{x}_{j,k},\addtag\label{eq:iterationv2}
\end{align*}
where $\hat{x}_{j,k}$ are given in~\eqref{citerationx}.

\begin{algorithm}[]
	\caption{PPDC Algorithm }
	\label{Al:PPDC}
	\begin{algorithmic}[1]
		\STATE \textbf{Input:} Stopping time $K$, adjacency matrix $W$, and positive parameters $\eta$, $\gamma$, $\omega$, $\alpha_x$, $s_{\xi_{x_i}}$, $s_{\xi_{v_i}}$, $q_i$, $\forall i\in\mathcal{V}$.
		\STATE \textbf{Initialization:} Each ~$i\in\mathcal{V}$ chooses arbitrarily $x_i(0)\in\mathbb{R}^d$, $x^c_i(0)=\bf{0}$, $v_i(0)=\bf{0}$.
		\FOR{$k=0,1,\dots,K-1$}
		\FOR {for $i\in\mathcal{V}$ in parallel} 
		\STATE Generate Laplace noises $\xi_{x_i,k}\sim\text{Lap}_d(s_{\xi_{x_i}}q_i^k)$ and $\xi_{v_i,k}\sim\text{Lap}_d(s_{\xi_{v_i}}q_i^k)$.
		\STATE Compute $C(x_{i,k}+\xi_{x_i,k}-x_{i,k}^c)$ and broadcast them to its neighbors $\mathcal{N}_i$.
		\STATE Receive $C(x_{j,k}+\xi_{x_j,k}-x_{j,k}^c)$from $j\in\mathcal{N}_i$.	
		\STATE Update $x_{i,k+1}$ and $v_{i,k+1}$ from~\eqref{eq:iterationxp2} and~\eqref{eq:iterationv2}, respectively.
		\STATE Update $x_{i,k}^c$ from~\eqref{citerationxc}.
		\ENDFOR
		\ENDFOR
		\STATE \textbf{Output:} \{$x_{i,k}$\}.
	\end{algorithmic}
\end{algorithm}

\subsection{Convergence Analysis of PPDC}
In this section, we first show the convergence of PPDC with and without P--L condition.

\begin{theorem}\label{theo:convergence1}
	Suppose Assumptions~\ref{as:smooth}--\ref{as:finite} and \ref{as:strongconnected}--\ref{as:compressor} hold, under Algorithm~\ref{Al:PPDC}, if~$\gamma=\tilde{\zeta}_1\omega,~\omega>\tilde{\zeta}_2$, $\alpha_x\in(0,\frac{1}{r})$, and $0<\eta<\tilde{\zeta}_3$, for any $T\in\mathbb{N}$, it holds that 
 \begin{align*}
    \frac{1}{T}\sum_{k=0}^T\big(\mathbb{E}\Vert\x_k-\bx_k\Vert^2&+\mathbb{E}\Vert \nabla f(\bar{x}_k)\Vert^2\big)\leq \frac{\bar{\kappa}_3M_2}{T}+{\check{\kappa}_1\mathbb{E}\Vert\sum_{t=0}^{\infty}\bar{\xi}_{v,t}\Vert^2},\addtag\label{eq:theo2}
 \end{align*}
 where $\tilde{\zeta}_1,\tilde{\zeta}_2,\tilde{\zeta}_3,\bar{\kappa}_3,\check{\kappa}_1$ and $M_2$ are positive constants given in Appendix~\ref{app-convergence1}.
\end{theorem}
\begin{proof}
	See Appendix~\ref{app-convergence1}.
\end{proof}
Similarly, we have the following result.
\begin{corollary}\label{coro1}
    Under the same assumptions and parameters in Theorem~\ref{theo:convergence1}. It holds that
\begin{align*}
\frac{1}{T}\sum_{k=0}^T\big(\mathbb{E}\Vert\x_k-\bx_k\Vert^2&+\mathbb{E}\Vert \nabla f(\bar{x}_k)\Vert^2\big)\leq \frac{\bar{\kappa}_3M_2}{T}+\frac{2d\check{\kappa}_1\bar{s}_{\xi,v}^2}{1-\bar{q}^2},
\end{align*}
where $\bar{s}_{\xi,v}=\max_i\{s_{\xi_{v_i}}\}$. 
\end{corollary}
\begin{remark}\label{remark:primal}
    Theorem~\ref{theo:convergence1} and Corollary~\ref{coro1} shows that PPDC converges to a neighborhood of a stationary point with the rate $\mathcal{O}(1/T)$ for general nonconvex cost functions, which is the same as PGTC. Furthermore, from Appendix~\ref{app-convergence1}, we have $\bar{\kappa}_3\geq\frac{4}{\eta}$ and $\kappa_{12}>\frac{\eta}{4}+2\eta\omega^2$. Since $\omega>\tilde{\zeta}_2>\tilde{\zeta}_4>1$, it holds that $n\check{\kappa}_1=\bar{\kappa}_3\kappa_{12}>9>4$. Recall Theorem~\ref{theo:convergence0}, compared the second term to the right side of~\eqref{eq:theo1} and~\eqref{eq:theo2}, it can be observed that PPDC seems to be more susceptible to noise than PGTC. As pointed out in~\cite{yi2021linear}, the primal-dual method is equivalent to the EXTRA algorithm proposed in~\cite{shi2015extra}. However, the EXTRA algorithm uses historical information to correct the difference between the local gradient and the global gradient. This implies that PPDC may accumulate additional noise compared with PGTC. Additionally, compared with PGTC, PPDC preserve stronger privacy, more details can be found in Remark~\ref{remark:privacy}.
\end{remark}
Then we provide the linear convergence of PPDC with the P--L condition.
\begin{theorem}\label{theo:convergence2}
	Suppose Assumption~\ref{as:smooth}--\ref{as:compressor} hold, under Algorithm~\ref{Al:PPDC}, if~$\gamma=\tilde{\zeta}_1\omega,~\omega>\tilde{\zeta}_2$, $\alpha_x\in(0,\frac{1}{r})$, and $0<\eta<\tilde{\zeta}_3$, we have
 \begin{align*}
     \E[\Vert\x_k-\bx_k\Vert^2&+n(f(\bar{x}_{k})-f^*)]\leq(1-\bar{\kappa}_4)^{k}M_3+\frac{2d\check{\kappa}_1\bar{s}_{\xi,v}^2}{1-\bar{q}^2},
 \end{align*}
 where $\tilde{\zeta}_1,\tilde{\zeta}_2,\tilde{\zeta}_3,\bar{\kappa}_4,\check{\kappa}_2$ and $M_3$ are positive constants given in Appendix~\ref{app-convergence2} with $0<\bar{\kappa}_4<1$.
\end{theorem}
\begin{proof}
	See Appendix~\ref{app-convergence2}.
\end{proof}

It is straightforward to see that PPDC linearly converges to a neighborhood of the optimum when the global cost function satisfies the P--L condition. By combining Theorems~\ref{theo:convergence1} and~\ref{theo:convergence2}, we know that PPDC has similar convergence property as PGTC. Furthermore, as we discussed before, PPDC requires fewer communication resources than PGTC. 
\subsection{Differential privacy}
In this section, we show that the differential privacy of all cost functions can be preserved under PPDC.
\begin{proposition}\label{propo:privacy1}
	Suppose Assumptions~\ref{as:smooth}--\ref{as:distance1} hold. PPDC preserves the $\epsilon_{i_0}$-differential privacy for any agent $i_0$'s cost function if $\eta<\frac{1}{2L_f}$ and $q_{i_0}\in(\frac{\eta L_f+\sqrt{\eta^2L_f^2+4\eta L_f}}{2},1)$, where $\epsilon_{i_0}$ is given by
	\begin{align}
	\epsilon_{i_0}=\frac{\tau_{i_0} q_{i_0}^2\delta}{q_{i_0}^2-\eta L_f-q_{i_0}\eta L_f},~~~~\forall i_0\in\mathcal{V}
	\end{align}
	with $\tau_{i_0}=\frac{1}{s_{\xi_{x_{i_0}}}}+\frac{1}{\eta s_{\xi_{v_{i_0}}}}$.
\end{proposition}
\begin{proof}
	The proof can be obtained in the same way as the proof of Theorem~2 in\cite{xie2023compressed}.
\end{proof}

Similarly, we use the same notation in~Theorem~\ref{theo:privacy} and provide the following theorem.
\begin{theorem}\label{theo:privacy1}
	Suppose Assumption~\ref{as:distance} holds, given a finite number of iterations $K$, PPDC preserves the $\epsilon_{i_0}$-differential privacy for any agent $i_0$'s cost function if the parameters satisfy
\begin{align*}
2\sqrt{d}M\sum_{k=0}^K\left(\frac{\sqrt{\eta}}{s_{\xi_{x_{i_0}}}q_{i_0}^k}+\frac{2}{\omega s_{\xi_{v_{i_0}}}q_{i_0}^k}\right)\leq \epsilon_{i_0}.\addtag\label{eq:theo4}
\end{align*}
\end{theorem}
\begin{proof}
	See Appendix~\ref{app-privacy1}.
\end{proof}
\begin{remark}\label{remark:privacy}
    Suppose the parameter of noise $s_{\xi_{v_i}}=s_{\xi_{y_i}}$, since $\omega>1$, it can be seen that the left side of~\eqref{eq:theo4} is less than or equal to the left side of~\eqref{eq:theo3}. This means that the condition~\eqref{eq:theo4} holds more easily than~\eqref{eq:theo3} for a given $\epsilon_{i_0}$. In other words, the privacy under PPDC is more stronger than PGTC under the same noises parameters.
\end{remark}
\subsection{Proof Sketch}
We then provide the proof sketch of Theorems~\ref{theo:convergence1} and~\ref{theo:convergence2}. Similar to the proof of PGTC, we track the compressed errors and consensus errors of state $\bx_k$ and dual state $\bv_k$ by auxiliary function $\widetilde{V}_{k}$, which is defined in Appendix~\ref{app-convergence1}. To estimate those errors, we construct a linear inequality of $\widetilde{V}_{k}$, which is stated in Lemma~\ref{lemma:lyapunov}. Notice that noises also accumulate on the dual variable, i.e.
\begin{align}\label{eq:dual}   \bv_{k+1}=\bv_k+\sum_{t=0}^{k}\bar{\xi}_{v,t}.
\end{align}

Similar to PGTC, we distinguish the redundant parts $\sum_{t=0}^\infty\xi_{v,t}$. Then we show that the errors actually decrease, and provide some sufficient parameters leading to the claimed convergence results in Theorems~\ref{theo:convergence1} and~\ref{theo:convergence2}. Finally, the proofs of Theorems~\ref{theo:convergence1} and~\ref{theo:convergence2} are inspired by the proofs of Theorems 1 and 2 in~\cite{yi2022communication}. However, since this paper considers the privacy, we need to analyze the noise and deal with the redundant parts~\eqref{eq:dual}. 


\section{simulation}\label{simulation}
\begin{center}
	\begin{figure}
		\center
		\usetikzlibrary{positioning, fit, calc}

\begin{tikzpicture}[->,>=stealth',shorten >=1pt,auto,node distance=3.5cm,
thick,main node/.style={circle,fill=gray!20,draw,font=\sffamily\normalsize\bfseries}]

\node[main node] (1) {1};
\node[main node] (2) [ right of=1] {2};
\node[main node] (3) [ below right of=2] {3};
\node[main node] (4) [ below left of=3] {4};
\node[main node] (5) [ left of=4] {5};
\node[main node] (6) [ below left of=1] {6};

%
%

\path[every node/.style={font=\sffamily\small}]
(1) edge  node[] {} (4)
edge  node[left] {} (2)	  
edge   node[above]{} (6)
(2) edge  node[left] {} (5)
edge  node[above]{} (1)
edge  node[left] {} (3)
(3) edge  node[left] {} (4)
edge  node[left] {} (2)
(4) edge  node[above]{} (3)
edge  node[above]{} (1)	
edge  node[left] {} (5)
(5) edge  node[left] {} (2)
edge node[left] {} (4)
edge node[left] {} (6)
(6) edge node[above]{} (5)
edge node[above]{} (1);
\end{tikzpicture}
		\caption{A connected undirected graph consisting of 6 agents.}
		\label{Fig:graph}
	\end{figure}
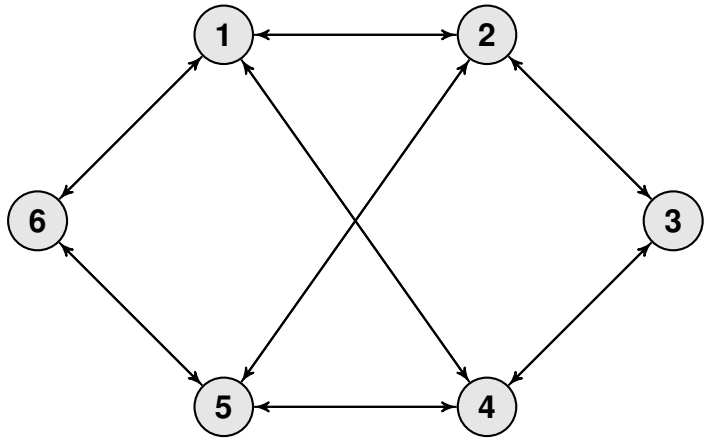
\end{center}
In this section, simulations are given to verify the validity of PGTC and PPDC. We first consider the following three compressors:
\begin{itemize}
	\item Greedy (Top-$k$) quantizer~\cite{beznosikov2020biased}:
	\begin{align*}
		C_1(x):=\sum_{i_s=1}^kx_{(i_s)}e_{i_s},
	\end{align*}
where $x_{(i_s)}$ is the $i_s$-th coordinate of $x$ with $i_1,\dots,i_k$ being the indices of the largest $k$ coordinates in magnitude of $x$, and $e_1,\dots,e_d$ are the standard unit basis vectors in $\R^d$.
    \item Biased $b$-bits quantizer~\cite{koloskova2019decentralized}:
    \begin{align*}
    	C_2(x):=\frac{\Vert x\Vert}{\xi}\cdot \text{sign}(x)\cdot 2^{-(b-1)}\circ \left\lfloor\frac{2^{(b-1)}\vert x\vert}{\Vert x\Vert} +u\right\rfloor,
    \end{align*}
where $\xi=1+\min\{\frac{d}{2^{2(b-1)}},\frac{\sqrt{d}}{2^{(b-1)}}\}$, $u$ is a random dithering vector uniformly sampled from $[0,1]^d$, $\circ$ is the Hadamard product, and $\text{sign}(\cdot)$, $|\cdot|$, $\lfloor\cdot\rfloor$ are the element-wise sign, absolute and floor functions, respectively. 
\item Norm-sign compressor~\cite{yi2022communication}:
    \begin{align*}
    	C_3(x):=\frac{\Vert x\Vert_\infty}{2}\text{sign} (x).
    \end{align*}

\end{itemize}	
As pointed out in~\cite{koloskova2019decentralized}, all of the above three compressors satisfy Assumption~\ref{as:compressor}. Specifically, we choose $k=2$ and $b=2$ in the following simulations. 

We then consider two distributed nonconvex optimization problem with $n=6$ agents and they communicate on a connected undirected graph, whose topology is shown in Fig.~\ref{Fig:graph}. Specifically, we firstly assume the agents aims to slove the following nonconvex distributed binary classification problem~\cite{antoniadis2011penalized,yi2021linear,sun2019distributed}
\begin{align*}
	&~~~~~\min_{x}f(x)=\frac{1}{6}\sum_{i=1}^6 f_i,\\\addtag\label{eq:simulation0}
 &f_i\left(x_i\right)=\frac{1}{m} \sum_{j=1}^m \log \left(1+\exp \left(-u_{i j} x_i^{\top} v_{i j}\right)\right)+\sum_{s=1}^d \frac{\lambda \alpha x_{i, s}^2}{1+\alpha x_{i, s}^2}.
\end{align*}
where $v_{ij}\in\R^d$ is feature vector and randomly generated with standard Gaussian distribution $N(0,1)$, $u_{ij}\in\{-1,1\}$ is the label and randomly generated with uniformly distributed pseudorandom integers taking the values $\{-1,1\}$ and $x_{i, s}$ is the $s$-th coordinate of $x_i$. Specifically, we assume $\lambda=0.001,\alpha=1,m=200$ and the initial value of each agent $x_i(0)$ is randomly chosen in $[0,1]^{10}$. 
\begin{figure}[htbp]
	\centering
	\includegraphics[width=0.75\linewidth]{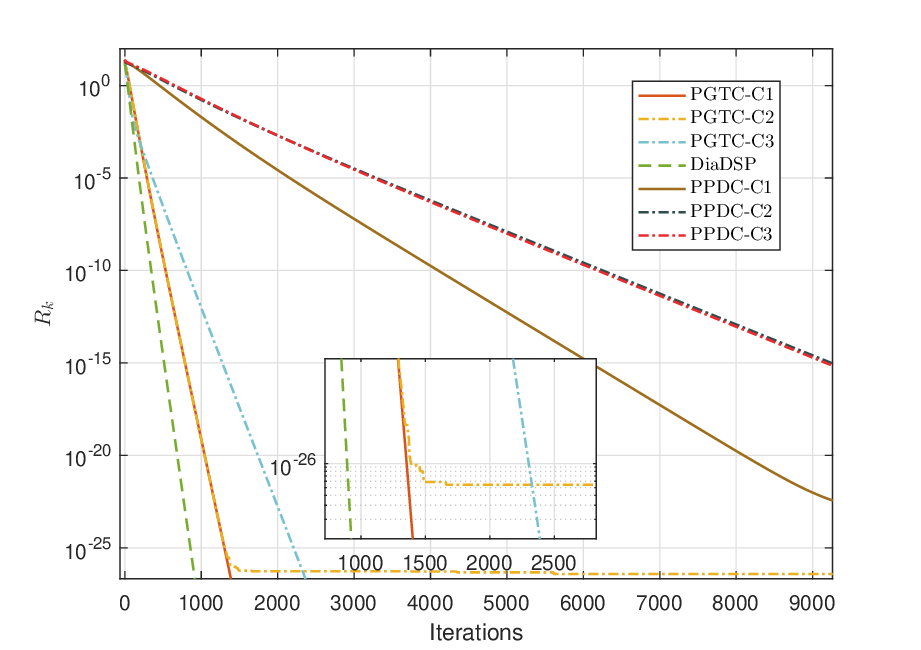}
	\caption{The evolution $R_k$ of residual under PGTC, PPDC, uncompressed method~DiaDSP for distributed binary classification problem~\eqref{eq:simulation0}.}
	\label{fig:convergence1}
\end{figure}
    \begin{table*}[htbp]
	\centering
	\resizebox{.7\linewidth}{!}{\begin{tabular}{lc c c c c c c c}\hline
		Algorithm & Compressor & $\gamma$ & $\omega$ & $\eta$ & $s_{\xi}$ & $q$ & $\alpha_x$ & $\alpha_y$\\\hline
		PGTC-C1 &$C_1$&0.2&---&0.1&100&0.1&0.5&0.5\\
		PGTC-C2 &$C_2$&0.2&---&0.1&100&0.1&0.5&0.5\\
		PGTC-C3 &$C_3$&0.1&---&0.15&100&0.1&0.5&0.5\\
		DiaDSP &---&---&---&0.15&100&0.1&---&---\\
  	PPDC-C1 &$C_1$&45&5&0.015&100&0.1&0.2&---\\
		PPDC-C2 &$C_2$&45&5&0.01&100&0.1&0.2&---\\
		PPDC-C3 &$C_3$&25&5&0.01&100&0.1&0.2&---\\\hline
	\end{tabular}}
	\caption{Parameter setting for different algorithms.}
	\label{tab:parameter}
\end{table*}

\begin{figure}[htbp]
	\centering
	\includegraphics[width=0.75\linewidth]{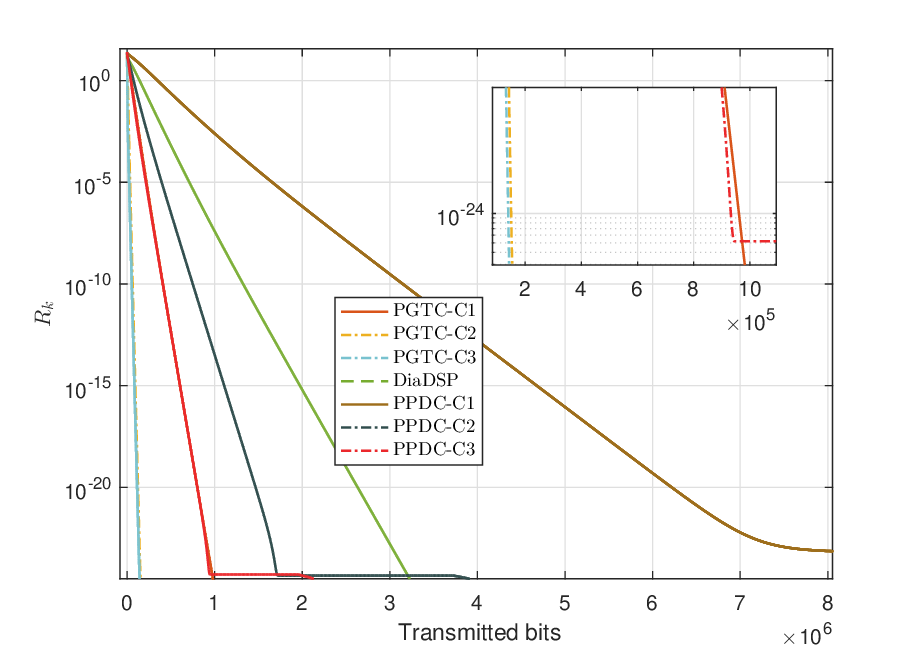}
	\caption{The evolution $R_k$ of residual with respect to the transmitted bits under PGTC, PPDC, uncompressed method~DiaDSP for distributed binary classification problem~\eqref{eq:simulation0}.}
	\label{fig:bit1}
\end{figure}
\begin{figure}[htbp]
	\centering
	\includegraphics[width=0.75\linewidth]{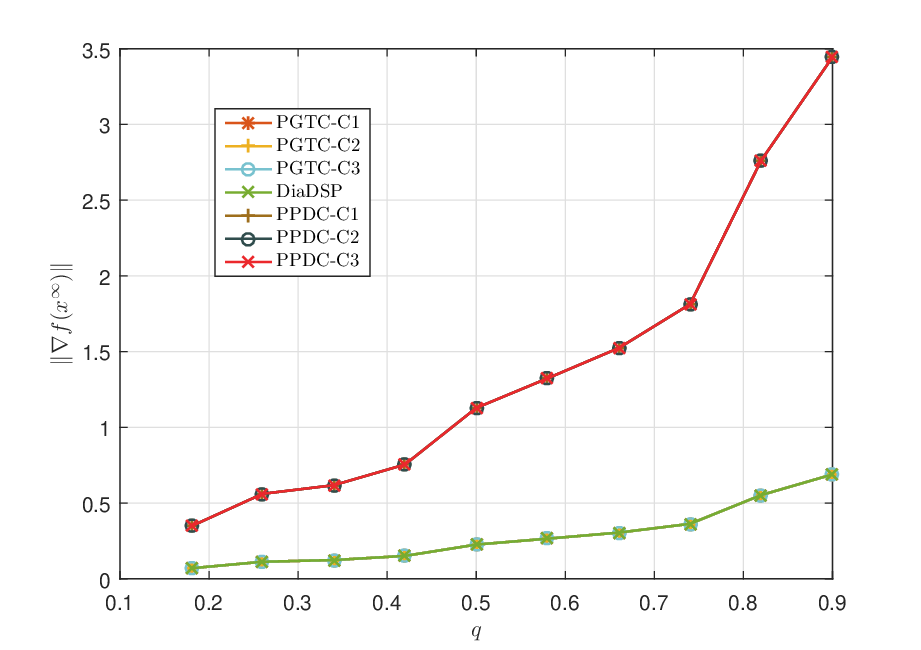}
	\caption{Effect of noise decaying rate on convergence accuracy for distributed binary classification problem~\eqref{eq:simulation0}.}
	\label{fig:accuracy0}
\end{figure}
\begin{figure}[htbp]
	\centering
	\includegraphics[width=0.75\linewidth]{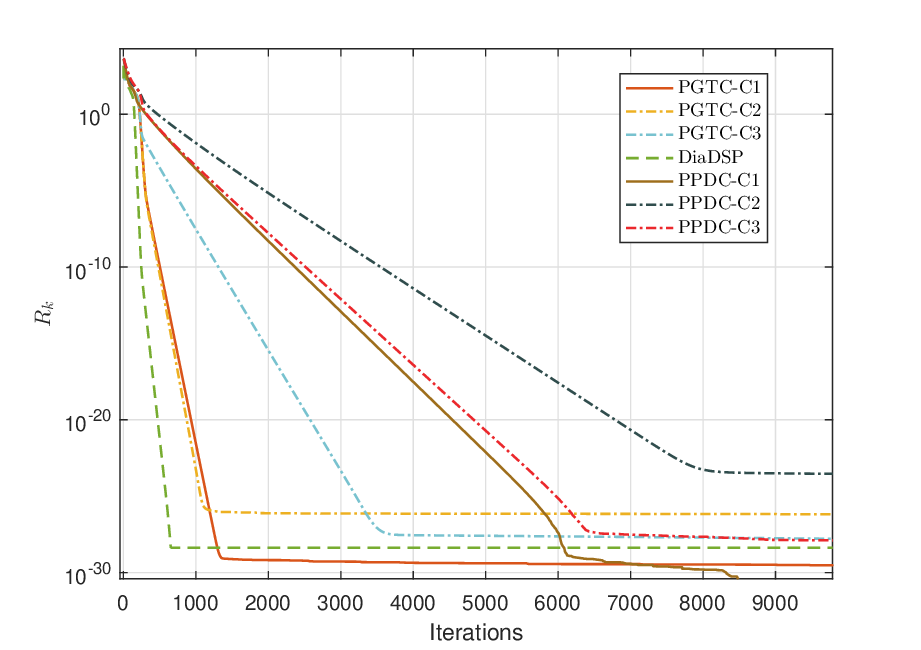}
	\caption{The evolution $R_k$ of residual under PGTC, PPDC, uncompressed method~DiaDSP for distributed nonconvex optimization problem~\eqref{eq:simulation}.}
	\label{fig:convergence2}
\end{figure}
We conduct experiments to verify the convergence rate of PGTC and PPDC using different compressors. The parameters are set as specified in TABLE~\ref{tab:parameter}, except for $s_{\xi}=0.1$ and $q=0.2$, which are consistent across all agents ($s_{\xi_{x_i}}=s_{\xi_{y_i}}=s_{\xi_{v_i}}=s_{\xi}$ and $q_i=q$ for all $i\in\mathcal{V}$). To evaluate the convergence, we compute the residual defined as $R_k\triangleq\min_{t\leq k}\Vert\x_t-\x^\infty\Vert^2$, where $\x^\infty$ is the convergence point. Fig.~\ref{fig:convergence1} shows that $\x_k$ linearly converges to the point $\x^\infty$ under PGTC and PPDC with different constant stepsizes and compressors. Furthermore, the convergence rate of PGTC can closely match that of DiaDSP~\cite{ding2021differentially} with suitable parameters and compressors. Fig.~\ref{fig:bit1} illustrates that, compared with DiaDSP, most of our algorithms require less number of bits. This means our methods are more efficient. We then simulate the effect of the noise decaying rate on convergence accuracy. We use $\Vert \lf(x^\infty)\Vert$ to measure the convergence accuracy of different algorithms. We set $s_\xi=0.1$ and other parameters are the same as TABLE~\ref{tab:parameter}. The relation between accuracy and decaying rate $q$ is shown in Fig.~\ref{fig:accuracy0}, where $q=0.18,0.26,0.34,0.42,0.5,0.58,0.66,0.74,0.82,0.9$. It can be seen that the accuracy of PGTC is nearly the same as that of DiaDSP and the accuracy is only noise dependent and not related to stepsize, $\gamma$, and compressors. Compared with PGTC, it can also be seen that the convergence accuracy of PPDC is more affected by noise.

We further consider the following nonconvex problem~\cite{karimi2016linear}
\begin{align*}
	&~~~~~\min_{x}f(x)=\frac{1}{6}\sum_{i=1}^6 f_i,\\\addtag\label{eq:simulation}
 &f_i(x)=x^\top x+3sin(x)^\top sin(x)+m_ix^\top cos(x),
\end{align*}
where $m_i\in\R$ is constant. In this example, the parameter $m_i$ is randomly generated and such that $\sum_{i=1}^6m_i=0,~m_i\neq0,~\forall i\in\mathcal{V}$. The initial value of each agent $x_i(0)$ is randomly chosen in $[0,1]^{10}$.

\begin{figure}[t]
	\centering
	\includegraphics[width=0.75\linewidth]{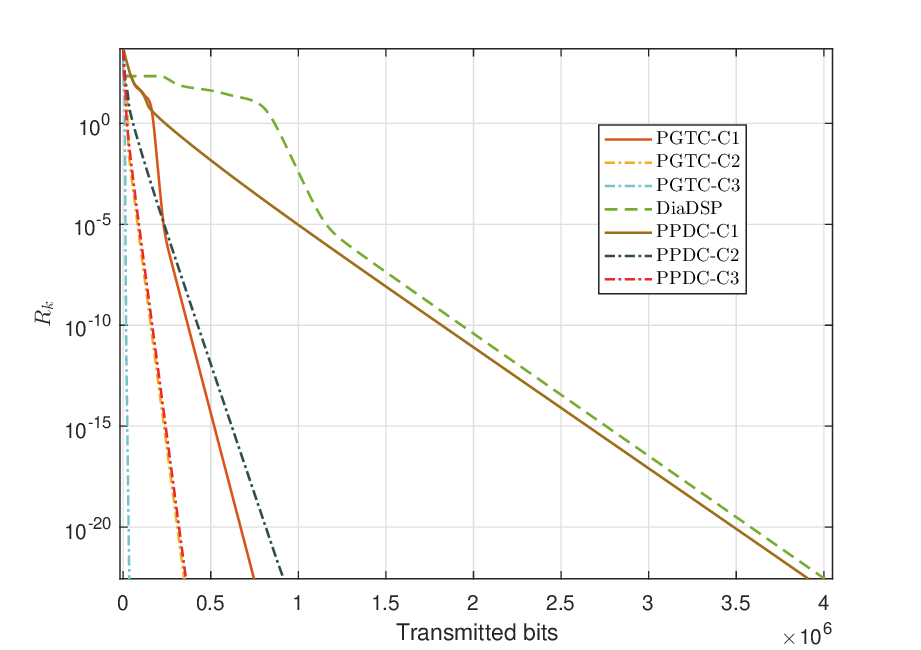}
	\caption{The evolution $R_k$ of residual with respect to the transmitted bits under PGTC, PPDC, uncompressed method~DiaDSP for distributed nonconvex optimization problem~\eqref{eq:simulation}.}
	\label{fig:bit2}
\end{figure}
\begin{figure}[htbp]
	\centering
	\includegraphics[width=0.75\linewidth]{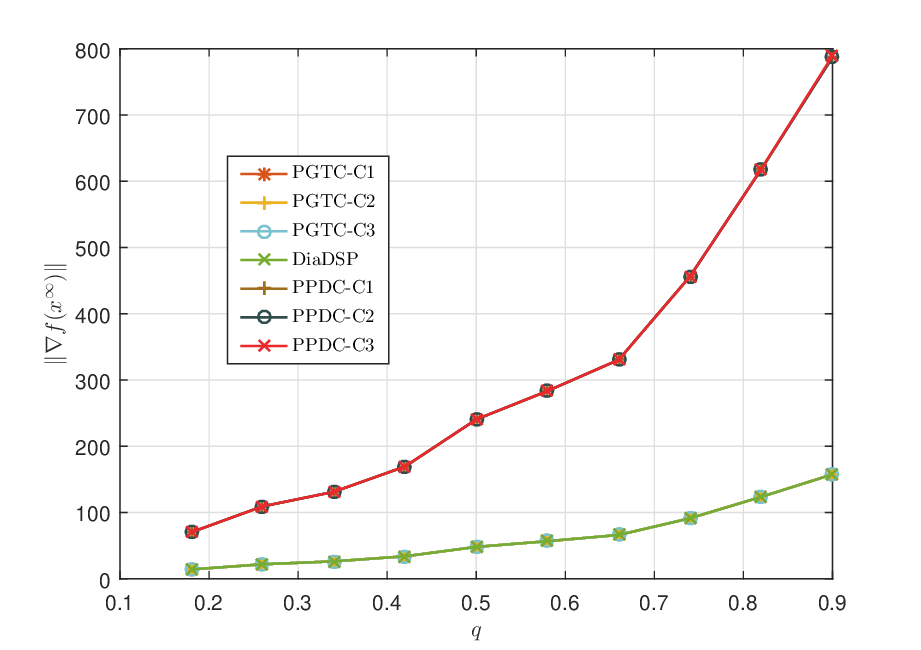}
	\caption{Effect of noise decaying rate on convergence accuracy for distributed nonconvex optimization problem~\eqref{eq:simulation}.}
	\label{fig:accuracy}
\end{figure}

Similar to Fig.~\ref{fig:convergence1}, we then verify the convergence rate of PGTC and PPDC with different compressors for distributed nonconvex optimization problem~\eqref{eq:simulation}, the parameters of different algorithms are given in TABLE~\ref{tab:parameter}. Fig.~\ref{fig:convergence2} shows that $\x_k$ linearly converges to the point $\x^\infty$ under PGTC and PPDC with different constant stepsize and compressors for problem~\eqref{eq:simulation}. As shown in Fig.~\ref{fig:convergence2}, due to the primal-dual method causes more noise redundancy~(details can be fond in Remark~\ref{remark:primal}), the PPDC is generally slower than PGTC. Fig.~\ref{fig:bit2} shows that compared with the ideal communication method, our algorithms converge to the same accuracy with much fewer bits transmitted. In addition, even from the perspective of transmitted bit, the PPDC still generally slower than PGTC. We then simulate the effect of the noise decaying rate on convergence accuracy for distributed nonconvex optimization problem~\eqref{eq:simulation}. Similar to the Fig.~\ref{fig:accuracy0}, let $s_\xi=5$ and other parameters be the same as TABLE~\ref{tab:parameter}. The relation between accuracy and decaying rate $q$ is shown in Fig.~\ref{fig:accuracy}. It can be seen that the accuracy of our methods is only noise dependent and PPDC is more susceptible to noise compared to PGTC. This also verifies the Remark~\ref{remark:primal}.

\section{conclusion}\label{conclusion}
In this paper, we investigated differentially private distributed nonconvex optimization under limited communication. Specifically, we proposed two algorithms under compressed communication. We established sublinear convergence for smooth (possibly nonconvex) cost functions and linear convergence when the global cost functions additionally satisfy the Polyak-Łojasiewicz condition  even for a general class of compressors with bounded relative compression error. Furthermore, we observed that the proposed algorithms achieve similar accuracy to the algorithm with idealized communication. Importantly, compared with existing literature, our proposed algorithms preserve a more rigorous $\epsilon$-differential privacy for the local cost function of each agent and are suitable for a general class of compressors. Future work includes extending the study to directed graphs and exploring the relationship between compressors and privacy performance.

\appendices
\section{The proof of Lemma~\ref{lemma:linearinequalities}}\label{app-convergence00}
\subsection{Supporting Lemmas}
We first introduce some useful vector and matrix inequalities.
\begin{lemma}\label{lemma:usefulinequalities}
For $u,v\in\R^d$, and $\forall s>0$ we have
\begin{align*}
	&~~~~~~u^\top v\leq \frac{s}{2}\Vert u\Vert^2+\frac{1}{2s}\Vert v\Vert^2,\addtag\label{eq:inequality1}\\
	&\Vert u+v\Vert^2\leq(1+s)\Vert u\Vert^2+(1+\frac{1}{s})\Vert v\Vert^2.\addtag\label{eq:inequality2}
\end{align*}
\end{lemma}
\begin{lemma}\label{lemma:upperboundoflw}
	\cite{liao2022compressed} Suppose Assumption~\ref{as:strongconnected} holds. For~$\gamma\in(0,1]$ and any $\omega\in\R^{nd}$, we have $\left\Vert W_\gamma\omega-\bar{\omega}\right\Vert\leq\hat{\lambda}\left\Vert\omega-\bar{\omega}\right\Vert$, where $\hat{\lambda}=1-\gamma(1-\llw)$ with $\llw=\bar{\lambda}_{W-\frac{\mathbf{1}\mathbf{1}^\top}{n}}$.
	\end{lemma}
 \begin{lemma}
     Suppose an random variable $x\sim\text{Lap}(\theta)$, we have $\E[x^2]=2\theta^2$ and $\E[\vert x\vert]=\theta$.
 \end{lemma}
\subsection{The proof of Lemma~\ref{lemma:linearinequalities}}
Denote $\bh=\frac{1}{n}(\mathbf{1}_n\mathbf{1}_n^\top\otimes\bi_d)$, $\bk=K_n\otimes\bi_d=\bi_{nd}-H$. We then prove Lemma~\ref{lemma:linearinequalities} by constructing the upper bounds of $\mathbb{E}[\oo_{x,k+1}],~\mathbb{E}[\oo_{y,k+1}],~\mathbb{E}[\oo_{\sigma_x,k+1}]$, and $\mathbb{E}[\oo_{\sigma_y,k+1}]$, respectively.

(a) According to~\eqref{iterationx1}, we obtain 
\begin{align*}
	\E&[\oo_{x,k+1}]=\E[\Vert W_\gamma\x_k-\bx_k+W_\gamma\nxk-\nbxk\\
	&+\gamma(W-I_n)\otimes I_d(\hx_k-\x_k-\nxk)-\eta(\y_k-\by_k)\Vert^2]\\
	&\leq(1+s)(1-\gamma\rho)\E[\oxk]\!+\!(1+\frac{1}{s})(3\lwi^2\gamma^2r_0\E[\osxk]\\
	&+3\eta^2\E[\oyk]+3\hat{\lambda}^2\E[\Vert\nxk-\nbxk\Vert^2])\\
	&\leq(1-\frac{\gamma\rho}{2})\E[\oxk]+\frac{9\lwi^2\gamma r_0}{\rho}\E[\osxk]+\frac{9\eta^2}{\gamma\rho}\E[\oyk]+\mu_{1,k},\addtag\label{eq:boundofaveandstatex}
\end{align*}
where $\mu_{1,k}=\frac{9\hat{\lambda}^2}{\gamma\rho}\E[\Vert\nxk-\nbxk\Vert^2]$; the first inequality holds comes from~\eqref{eq:propertyofcompressors1},~\eqref{eq:inequality2}, and Lemma~\ref{lemma:upperboundoflw}, and denoting $\rho=1-\llw$; the second inequality holds by choosing $s=\gamma\rho/2$ and $\gamma\leq1$, $\rho\leq1$. Then we constructed the relationship between~$\E[\oo_{x,k+1}]$ and~$\E[\oo_{x,k}]$.

(b) From~\eqref{iterationy1}, we have
\begin{align*}
	\E&[\oo_{y,k+1}]=\E[\Vert W_\gamma\y_k-\by_k+W_\gamma\nyk-\nbyk\\
	&+\gamma(W-I_n)\otimes I_d(\hy_k-\y_k-\nyk)\\
	&+\bk(\f(\x_{k+1})-\f(\x_k))\Vert^2]\\
	&\leq (1-\frac{\gamma\rho}{2})\E[\oyk]+\frac{9\lwi^2\gamma r_0}{\rho}\E[\osyk]\\
	&+\frac{9L_f^2}{\gamma\rho}\E[\Vert\x_{k+1}-\x_k\Vert^2]+\frac{9\hat{\lambda}^2}{\gamma\rho}\E[\Vert\nyk-\nbyk\Vert^2],\addtag\label{eq:boundofy}
\end{align*}
where the first inequality due to~\eqref{eq:propertyofcompressors1}, $\bar{\lambda}_\bk=1$, Assumption~\ref{as:smooth}, and Lemma~\ref{lemma:upperboundoflw}. By~\eqref{iterationx}, it holds that
\begin{align*}
\E[&\Vert\x_{k+1}-\x_k\Vert^2]=\E[\Vert\nxk+\gamma(W-I_n)\otimes I_d\hx_k-\eta\y_k\Vert^2]\\
&\leq\E[\Vert\gamma(W-I_n)\otimes I_d(\hx_k-\x_k-\nxk)\\
&+\gamma(W-I_n)\otimes I_d(\x_k-\bx_k)+W_\gamma\nxk-\eta\y_k\Vert^2]\\
&\leq 4\gamma^2\lwi^2r_0\E[\osxk]+4\gamma^2\lwi^2\E[\oxk]\\
&+4\eta^2\E[\oyk]+4\eta^2\E[\Vert\by_k\Vert^2]+4\hat{\lambda}^2\E[\Vert\nxk\Vert^2].\addtag\label{eq:boundofstatex}
\end{align*}
where the second inequality holds due to~\eqref{eq:propertyofcompressors1} and the fact that $\Vert \y\Vert^2=\Vert \y-\by\Vert^2+\Vert\by\Vert^2$ for any vector $\y\in\R^{nd}$. Combining~\eqref{eq:boundofy}--\eqref{eq:boundofstatex}, one obtains that
\begin{align*}
\E[\oo&_{y,k+1}]\leq (1-\frac{\gamma\rho}{2}+\frac{36\eta^2L_f^2}{\gamma\rho})\E[\oyk]\\
&+\frac{9\lwi^2\gamma r_0}{\rho}\E[\osyk]+\frac{36\gamma\lwi^2 L_f^2}{\rho}\E[\oxk]\\
&+\frac{36\gamma\lwi^2 L_f^2}{\rho}\E[\osxk]+\frac{36\eta^2 L_f^2}{\gamma\rho}\E[\Vert\by_k\Vert^2]+\mu_{2,k},\addtag\label{eq:boundofaveandstatey}
\end{align*}
where $\mu_{2,k}=\frac{36\hat{\lambda}^2 L_f^2}{\gamma\rho}\E[\Vert\nxk\Vert^2]+\frac{9\hat{\lambda}^2}{\gamma\rho}\E[\Vert\nyk-\nbyk\Vert^2]$. Then we constructed the relationship between~$\E[\oo_{y,k+1}]$ and~$\E[\oo_{y,k}]$.

(c) We have
\begin{align*}
	\E[\oo&_{\sigma_x,k+1}]=\E[\Vert \x^c_{k+1}-\x_{k+1}-\nxkk\Vert^2]\\
	&=\E[\Vert \x_{k}-\x_{k+1}+\nxk-\nxkk+\x^c_k-\x_k-\nxk\\
	&+\alpha_xr\frac{1}{r} C(\x_k+\nxk-\x_k^c)\Vert^2]\\
	&\leq (1+s)(\alpha_xr(1-\varphi)+(1-\alpha_xr))\E[\osxk]\\
	&+(1+\frac{1}{s})(2\E[\Vert\x_{k+1}-\x_k\Vert^2]+2\E[\Vert\nxkk-\nxk\Vert^2])\\
	&\leq (1-\frac{\varphi_1}{2})\E[\osxk]+\frac{4}{\varphi_1}\E[\Vert\x_{k+1}-\x_k\Vert^2]\\
	&+\frac{4}{\varphi_1}\E[\Vert\nxkk-\nxk\Vert^2]\\
	&\leq(1-\frac{\varphi_1}{2}+\frac{16\gamma^2\lwi^2r_0}{\varphi_1})\E[\osxk]\\
	&+\frac{16\gamma^2\lwi^2}{\varphi_1}\E[\oxk]+\frac{16\eta^2}{\varphi_1}\E[\oyk]+\frac{16\eta^2}{\varphi_1}\E[\Vert\by_k\Vert^2]+\mu_{3,k},\addtag\label{eq:boundofsigmax}
\end{align*}
where $\mu_{3,k}=\frac{16\hat{\lambda}^2}{\varphi_1}\E[\Vert\nxk\Vert^2]+\frac{4}{\varphi_1}\E[\Vert\nxkk-\nxk\Vert^2]$; the second equality comes from~\eqref{citerationx} and~\eqref{citerationxc}; the first inequality comes from~Lemma~\ref{lemma:usefulinequalities} and Jensen's inequality; the second inequality follows by denoting $\varphi_1=\min\{\alpha_xr\varphi,\alpha_yr\varphi\}$, choosing $s=\frac{\varphi_1}{2-2\varphi_1}$, and $\alpha_xr<1$. Then we constructed the relationship between~$\E[\oo_{\sigma_x,k+1}]$ and~$\E[\oo_{\sigma_x,k}]$.

(d) Similar to~\eqref{eq:boundofsigmax}, we have
\begin{align*}
	\E[\oo&_{\sigma_y,k+1}]\leq (1-\frac{\varphi_1}{2})\E[\osyk]+\frac{4}{\varphi_1}\E[\Vert\y_{k+1}-\y_k\Vert^2]\\
	&+\frac{4}{\varphi_1}\E[\Vert\nykk-\nyk\Vert^2].\addtag\label{eq:boundofsigmay0}
\end{align*}
From~\eqref{iterationy}, it holds that
\begin{align*}
	\E[&\Vert\y_{k+1}-\y_k\Vert^2]=\E[\Vert\nyk+\gamma(W-I_n)\otimes I_d\hy_k\\
 &+\f(\x_{k+1})-\f(\x_k)\Vert^2]\\
&\leq\E[\Vert\gamma(W-I_n)\otimes I_d(\hy_k-\y_k-\nyk)\\
&+\gamma(W-I_n)\otimes I_d(\y_k-\by_k)+W_\gamma\nyk+\f(\x_{k+1})\\
&-\f(\x_k)\Vert^2]\\
&\leq 4\gamma^2\lwi^2r_0\E[\osyk]+(4\gamma^2\lwi^2+16\eta^2L_f^2)\E[\oyk]\\
&+16\gamma^2\lwi^2L_f^2r_0\E[\osxk]+16\gamma^2\lwi^2L_f^2\E[\oxk]\\
&+16\eta^2L_f^2\E[\Vert\by_k\Vert^2]+16\hat{\lambda}^2\E[\Vert\nxk\Vert^2]+4\hat{\lambda}^2\E[\Vert\nyk\Vert^2].\addtag\label{eq:boundofstatey}
\end{align*}
where the first inequality holds due to the fact that $(W-I_n)\otimes I_d\by=0$. Combining~\eqref{eq:boundofsigmay0}--\eqref{eq:boundofstatey}, one obtains that
\begin{align*}
	\E[\oo&_{\sigma_y,k+1}]\leq (1-\frac{\varphi_1}{2}+\frac{16\gamma^2\lwi^2r_0}{\varphi_1})\E[\osyk]\\
	&+\frac{16\gamma^2\lwi^2+64\eta^2L_f^2}{\varphi_1}\E[\oyk]
	\\&+\frac{64\gamma^2\lwi^2L_f^2r_0}{\varphi_1}\E[\osxk]+\frac{64\gamma^2\lwi^2L_f^2}{\varphi_1}\E[\oxk]\\
	&+\frac{64\eta^2L_f^2}{\varphi_1}\E[\Vert\by_k\Vert^2]+\mu_{4,k},\addtag\label{eq:boundofsigmay}
\end{align*}
where $\mu_{4,k}=\frac{64\lwi^2}{\varphi_1}\E[\Vert\nxk\Vert^2]+\frac{16\lwi^2}{\varphi_1}\E[\Vert\nyk\Vert^2]+\frac{4}{\varphi_1}\E[\Vert\nykk-\nyk\Vert^2]$. Then we construct the relationship between~$\E[\oo_{\sigma_y,k+1}]$ and~$\E[\oo_{\sigma_y,k}]$. Let $\mu_k\!\triangleq\![\mu_{1,k},\mu_{2,k},\mu_{3,k},\mu_{4,k}]^\top$\!, combining~\eqref{eq:boundofaveandstatex},~\eqref{eq:boundofaveandstatey},~\eqref{eq:boundofsigmax}, and~\eqref{eq:boundofsigmay}, we have
\begin{align*}
		\mathbb{E}[\Theta_{k+1}]\preceq G\mathbb{E}[\Theta_k]+\vartheta_1\E[\left\Vert\by_k\right\Vert]^2+\mu_k,
	\end{align*}
where the elements of the matrix $G\in\mathbb{R}^{4\times 4}$ and vectors $\vartheta_1,\in\R^{4}$ correspond to the coefficients in \eqref{eq:boundofaveandstatex},~\eqref{eq:boundofaveandstatey},~\eqref{eq:boundofsigmax}, and~\eqref{eq:boundofsigmay}. Since  $\xi_{x_i}\!\sim\text{Lap}_d(s_{\xi_{x_i}}q_i^k)$ and $\xi_{y_i}\sim\text{Lap}_d(s_{\xi_{y_i}}q_i^k)$, we have $\mu_k\preceq \vartheta_2 \bar{q}^{2k}\bar{s_\xi}^2$, where $\vartheta_2$ is given by
\begin{align}\label{parameters}
\begin{aligned}
\vartheta_2=\bigg\{&\frac{9\hat{\lambda}^2}{\gamma\rho},~\frac{36\hat{\lambda}^2 L_f^2}{\gamma\rho}+\frac{9\hat{\lambda}^2}{\gamma\rho},~\frac{16(\hat{\lambda}^2+1)}{\varphi_1},\\
&\frac{64\lwi^2+16(\lwi^2+1)}{\varphi_1} \bigg\}2nd.
\end{aligned}
\end{align}
Then we know that~\eqref{eq:linearinequalities} holds.

\section{The proof of Theorem~\ref{theo:convergence0}}\label{app-convergence0}
For simplicity of the proof, we also denote some notations.
\begin{align*}
    &\bar{\zeta}_1=\min\{\frac{1}{8\lwi\sqrt{r_0}},\frac{n}{2L_f(\frac{n}{2L_f^2}+930)},\frac{1}{\frac{2+18n}{L_f}+576},\\
    &~~~\frac{1}{\frac{2+72nr_0}{L_f}+576+1024r_0},\frac{1}{\frac{2}{L_f}+144r_0},\frac{1}{2\varphi_1+176}\},\\
    &\bar{\kappa}_1=\frac{4}{\eta},~M_1=\mathbb{E}[V_{0}]+\sum_{k=0}^{\infty}(b_1\bar{q}^{2k}\bar{s_\xi}^2),\\
    &b_1=\mathbf{s}^\top\vartheta_2+(\frac{1}{\eta}+L_f)n.
\end{align*}
We first construct a upper bound of $\E[f(\bar{x}_{k+1})]-f^*$.
\begin{align*}
	\E[&f(\bar{x}_{k+1})]-f^*\leq \\
	&\E[f(\bar{x}_k)]-f^*+\E[\lf(\bar{x}_k)^\top(\frac{1}{n}(\mathbf{1}_n^\top\otimes I_d)(\nxk\\
 &-\eta\y_k))]+\frac{L_f}{2}\E[\Vert\frac{1}{n}(\mathbf{1}_n^\top\otimes I_d)(\nxk-\eta\y_k)\Vert^2]\\
	&\leq \E[f(\bar{x}_k)]-f^*-\frac{\eta}{2}\E[\Vert\lf(\bar{x}_k)\Vert^2]-\frac{\eta}{2}\E[\Vert\bar{y}_k\Vert^2]\\
	&+\frac{\eta}{2}\E[\Vert\lf(\bar{x}_k)-\bar{y}\Vert^2]+\frac{\eta}{4}\E[\Vert\lf(\bar{x}_k)\Vert^2]\\
	&+\frac{1}{\eta}\E[\Vert\frac{1}{n}(\mathbf{1}_n^\top\otimes I_d)\nxk\Vert^2]\\
 &+\frac{L_f}{2}\E[\Vert\frac{1}{n}(\mathbf{1}_n^\top\otimes I_d)(\nxk-\eta\y_k)\Vert^2].\addtag\label{eq:boundofaveandoptimal}
\end{align*}
From~\eqref{iterationy1}, we introduce a key property of PGTC, i.e., for $k\geq0$,
\begin{align}\label{eq:sumofy}
	n\bar{y}_k=n\nabla\bar{f}(\mathbf{x}_k)+(\mathbf{1}_n^\top\otimes I_d)\sum_{t=0}^{k-1}\mathbf{\xi}_{y,t}.
\end{align}
Then~\eqref{eq:boundofaveandoptimal} can be rewritten as
\begin{align*}
	\E[&f(\bar{x}_{k+1})]-f^*\leq \E[f(\bar{x}_k)]-f^*-\frac{\eta}{4}\E[\Vert\lf(\bar{x}_k)\Vert^2]\\
	&-\frac{\eta}{2}\E[\Vert\bar{y}_k\Vert^2]+\frac{\eta}{2}\E[\Vert\lf(\bar{x}_k)-\nabla\bar{f}(\mathbf{x}_k)\\
	&-\frac{1}{n}(\mathbf{1}^\top\otimes I_d)\sum_{t=0}^{k-1}\mathbf{\xi}_{y,t}\Vert^2]+\frac{1}{\eta}\E[\Vert\frac{1}{n}(\mathbf{1}^\top\otimes I_d)\nxk\Vert^2]\\
 &+L_f\E[\Vert\frac{1}{n}(\mathbf{1}^\top\otimes I_d)\nxk\Vert^2]+\eta^2L_f\E[\Vert\bar{y}_k\Vert^2]\\
	&\leq \E[f(\bar{x}_k)]-f^*-\frac{\eta}{4}\E[\Vert\lf(\bar{x}_k)\Vert^2]\\
 &~~-\frac{1}{n}(\frac{\eta}{2}-\eta^2L_f)\E[\Vert\by_k\Vert^2]\\
	&~~+\frac{\eta L_f^2}{n}\E[\Vert\bx_k-\x_k\Vert^2]+\frac{\eta}{n}\E[\Vert\sum_{t=0}^{k-1}\mathbf{\xi}_{y,t}\Vert^2]\\
	&~~+\frac{1}{n}(\frac{1}{\eta}+L_f)\E[\Vert\nxk\Vert^2],\addtag\label{eq:boundofaveandoptimal0}
\end{align*}
where the second inequality holds due to Assumption~\ref{as:smooth} and Jensen's inequality. From~\eqref{eqn:lyapunov},~\eqref{eq:boundofaveandoptimal0}, and Lemma~\ref{lemma:linearinequalities}, we have
\begin{align*}
	\E[V_{k+1}]&\leq\E[V_k]-\frac{\eta}{4}\E[\Vert\lf(\bar{x}_k)\Vert^2]-\frac{\zeta_1L_f\eta}{2n}\E[\oxk]\\
 &+(\mathbf{s}^\top G-(1-\frac{\eta}{2})\mathbf{s}^\top+\mathbf{c}^\top)\E[\Theta_k]\\
	&-\frac{1}{n}(\frac{\eta}{2}-\eta^2L_f-\mathbf{s}^\top\vartheta_1)\E[\Vert\by_k\Vert^2]+b_1\bar{q}^{2k}\bar{s_\xi}^2\\
	&+\frac{\eta}{n}\E\left[\left\Vert\sum_{t=0}^k\xi_{y,t}\right\Vert^2\right]
\end{align*}
where
\begin{align*}
	&~\mathbf{c}=\begin{bmatrix}
		\frac{\eta L_f^2}{n}&0&0&0
	\end{bmatrix}.
\end{align*}
Then the Theorem~\ref{theo:convergence0} can be proved if there exists some positive constants $\zeta_1$--$\zeta_4$ such that the following inequalities hold.
\begin{align*}
	&((1-\frac{\eta}{2})I-G^\top)\mathbf{s}-\mathbf{c}\succeq\mathbf{0},\addtag\label{eq:verify1}\\
	&\frac{1}{n}(\frac{\eta}{2}-\eta^2L_f-\mathbf{s}^\top\vartheta_1)\geq0.\addtag\label{eq:verify2}
\end{align*}
Since $\zeta_\gamma=\zeta_\eta\leq\bar{\zeta}_1$, we have $\zeta_\gamma\leq\frac{1}{8\lwi\sqrt{r_0}}$, $\zeta_\eta\leq\frac{1}{2\varphi_1+176}<\frac{1}{12}$. From $\gamma=\zeta_\gamma\rho\varphi_1,~\eta=\zeta_\eta\gamma\rho^2/L_f$, and $\rho\leq1$, we have
\begin{align*}
   &1-\frac{\gamma\rho}{2}+\frac{36\eta^2L_f^2}{\gamma\rho}\leq  1-\frac{\gamma\rho}{4},\\
&1-\frac{\varphi_1}{2}+\frac{16\gamma^2\lwi^2r_0}{\varphi_1}<1-\frac{\varphi_1}{4}.
\end{align*}
Then~\eqref{eq:verify1}--\eqref{eq:verify2} can be transfer to the following inequality

\begin{align}\label{eq:matrixform1}
\tiny{
		\begin{bmatrix}
			\frac{\gamma\rho L_f}{2n}-\frac{\zeta_\eta\gamma\rho^2}{2n}&-\frac{36\lwi^2\gamma\rho L_f}{n}&-\frac{16\zeta_\gamma\lwi^2\gamma\rho L_f}{n}&-\frac{64\zeta_\gamma\lwi^2\gamma\rho^3 L_f}{n}\\
			-\frac{9\zeta_\eta^2\gamma\rho^3}{nL_f}&\frac{\gamma\rho^3}{4nL_f}-\frac{\zeta_\eta\gamma\rho^4}{2nL_f^2}&-\frac{16\zeta_\eta^2\zeta_\gamma\gamma\rho^5}{nL_f}&-\frac{16\zeta_\gamma\gamma\rho^3(\lwi^2+4\zeta_\eta^2\rho^4)}{nL_f}\\
			-\frac{9\zeta_\gamma\lwi^2r_0\varphi_1L_f}{n}&-\frac{36\zeta_\gamma\lwi^2\rho^2\varphi_1L_f}{n}&\frac{\varphi_1L_f}{4n}-\frac{\zeta_\eta\zeta_\gamma\rho^3\varphi_1}{2n}&-\frac{64\zeta_\gamma^2\lwi^2\rho^4 r_0\varphi_1L_f}{n}\\
			0&-\frac{9\zeta_\gamma\lwi^2r_0\varphi_1\rho^2}{nL_f}&0&\frac{\varphi_1\rho^2}{4nL_f}-\frac{\zeta_\eta\zeta_\gamma\rho^5\varphi_1}{2nL_f^2}\\
			0&-\frac{36\zeta_\eta^2\gamma\rho^5}{nL_f}&-\frac{16\zeta_\eta^2\zeta_\gamma\gamma\rho^5}{nL_f}&-\frac{36\zeta_\eta^2\zeta_\gamma\gamma\rho^7}{nL_f}
		\end{bmatrix}
		\begin{bmatrix}
			\zeta_1\\
			\zeta_2\\
			\zeta_3\\
			\zeta_4
		\end{bmatrix}\succeq 
		\begin{bmatrix}
			\frac{\zeta_\eta\gamma\rho^2 L_f}{n}\\
			0\\
			0\\
			0\\
			\frac{1}{n}(-\frac{\zeta_\eta\gamma\rho^2}{2L_f}+\frac{\zeta_\eta^2\zeta_\gamma\gamma\rho^5\varphi_1}{L_f})
		\end{bmatrix},}
	\end{align}
	which can be rewritten as
	\begin{align}\label{eq:matrixform2}
 \small{
		\begin{bmatrix}
			1-\frac{\zeta_\eta\rho}{L_f}&-72\lwi^2&-32\zeta_\gamma\lwi^2&-128\zeta_\gamma\lwi^2\rho^2\\
			-36\zeta_\eta^2&1-\frac{2\zeta_\eta\rho}{L_f}&-64\zeta_\eta^2\zeta_\gamma\rho^2&-64\zeta_\gamma(\lwi^2+4\zeta_\eta^2\rho^4)\\
			-36\zeta_\gamma\lwi^2r_0&-144\zeta_\gamma\lwi^2\rho^2&1-\frac{2\zeta_\eta\zeta_\gamma\rho^3}{L_f}&-256\zeta_\gamma^2\lwi^2\rho^4 r_0\\
			0&-36\zeta_\gamma\lwi^2r_0&0&1-\frac{2\zeta_\eta\zeta_\gamma\rho^3}{L_f}\\
			0&-72\zeta_\eta\rho^3&-32\zeta_\eta\zeta_\gamma\rho^3&-72\zeta_\eta\zeta_\gamma\rho^5
		\end{bmatrix}
		\begin{bmatrix}
			\zeta_1\\
			\zeta_2\\
			\zeta_3\\
			\zeta_4
		\end{bmatrix}\succeq 
		\begin{bmatrix}
			2\zeta_\eta\rho\\
			0\\
			0\\
			0\\
			-1+2\zeta_\eta\zeta_\gamma\rho^3\varphi_1
		\end{bmatrix}}.
	\end{align}
 Since $\varphi_1<1,~\rho\leq1$, and~$\lwi^2\leq4$, it is easy to verify that \eqref{eq:matrixform2} holds if $\zeta_\eta=\zeta_\gamma\leq\bar{\zeta}_1$, $\zeta_1=\frac{n}{2L_f},~\zeta_3=1$, and $\zeta_2=\zeta_4=\zeta_\eta$. Then we have~\eqref{eq:verify1}--\eqref{eq:verify2} hold, which means
 \begin{align*}
 &\frac{1}{T}\sum_{k=0}^{T}\big(\mathbb{E}[\Vert\lf(\bar{x}_k\Vert^2]+\mathbb{E}[\Vert \x_k-\bx_k\Vert^2]\big)\\
 &~~~~~\leq\frac{\bar{\kappa}_1 M_1}{T}+\frac{4}{n}\mathbb{E}\left[\left\Vert\sum_{t=0}^\infty\xi_{y,t}\right\Vert^2\right].
\end{align*}

\section{The proof of Theorem~\ref{theo:convergence}}\label{app-convergence}
In this proof, in addition to the notations used in the proof of Theorem~\ref{theo:convergence0}, we also denote
\begin{align*}
    &\bar{\zeta}_2=\min\{\frac{1}{8\lwi\sqrt{r_0}},\frac{n}{2L_f(\frac{n\nu}{2L_f^2}+930)},\frac{1}{\frac{2\nu+18n}{L_f}+576},\\
    &~~~\frac{1}{\frac{2\nu+72nr_0}{L_f}+576+1024r_0},\frac{1}{\frac{2\nu}{L_f}+144r_0},\frac{1}{2\varphi_1+176}\}.
\end{align*}
From~\eqref{eq:boundofaveandoptimal0} and Assumption~\ref{as:PLcondition}, we have
\begin{align*}
	\E[&f(\bar{x}_{k+1})]-f^*\\
	&\leq (1-\frac{\eta\nu}{2})(\E[f(\bar{x}_k)]-f^*)-\frac{1}{n}(\frac{\eta}{2}-\eta^2L_f)\E[\Vert\by_k\Vert^2]\\
	&+\frac{\eta L_f^2}{n}\E[\Vert\bx_k-\x_k\Vert^2]+\frac{\eta}{n}\E[\Vert\sum_{t=0}^{k-1}\mathbf{\xi}_{y,t}\Vert^2]\\
	&+\frac{1}{n}(\frac{1}{\eta}+L_f)\E[\Vert\nxk\Vert^2],\addtag\label{eq:boundofaveandoptimal1}
\end{align*}
From~\eqref{eqn:lyapunov},~\eqref{eq:boundofaveandoptimal1}, and Lemma~\ref{lemma:linearinequalities}, we have
\begin{align*}
	\E[V_{k+1}]&\leq (1-\frac{\eta\nu}{2})\E[V_k]+\!(\mathbf{s}^\top G-(1-\frac{\eta\nu}{2})\mathbf{s}^\top\!\!\!+\mathbf{c}^\top)\E[\Theta_k]\\
	&-\frac{1}{n}(\frac{\eta}{2}-\eta^2L_f-\mathbf{s}^\top\vartheta_1)\E[\Vert\by_k\Vert^2]+b_1\bar{q}^{2k}\bar{s_\xi}^2\\
	&+\frac{\eta}{n}\E\left[\left\Vert\sum_{t=0}^k\xi_{y,t}\right\Vert^2\right]
\end{align*}
Then the Theorem~\ref{theo:convergence} can be proved if there exists some positive constants $\zeta_1$--$\zeta_4$ such that the following inequalities hold.
\begin{align*}
	&((1-\frac{\eta\nu}{2})I-G^\top)\mathbf{s}-\mathbf{c}\succeq\mathbf{0},\addtag\label{eq:verify3}\\
	&\frac{1}{n}(\frac{\eta}{2}-\eta^2L_f-\mathbf{s}^\top\vartheta_1)\geq0.\addtag\label{eq:verify4}
\end{align*}
Similar to the proof of Theorem~\ref{theo:convergence0}, we complete the proof if $\zeta_\eta=\zeta_\gamma\leq\bar{\zeta}_2$, $\zeta_1=\frac{n}{2L_f},~\zeta_3=1$, and $\zeta_2=\zeta_4=\zeta_\eta$.

\section{The proof of Theorem~\ref{theo:privacy}}\label{app-privacy}

From PGTC, it is clear that the observation sequence $\mathcal{H}=\{\mathcal{H}_k\}_{k=0}^\infty$ is uniquely determined by the noise sequences $\xi_x=\{\nxk\}_{k=0}^\infty$, $\xi_y=\{\nyk\}_{k=0}^\infty$, and random sequence~$\mathbf{\varrho}=\{\mathbf{\varrho}_k\}_{k=0}^\infty$, where $\mathbf{\varrho}_k\in\R^{nd}$ is a vector and its element $[\mathbf{\varrho}_k]_{ij}$ is the compression perturbation of $x_{ij,k}^a-x^c_{ij,k}$. We use function $Z_{\mathcal{F}}$ to denote the relation, i.e., $\mathcal{H}=Z_{\mathcal{F}}(\xi_x,\xi_y,\mathbf{\varrho})$, where $\mathcal{F}=\{\x(0),W,\mathcal{S}\}$. From Definition~\ref{def:differentialprivacy}, to show the differential privacy of the cost function $f_{i_0}$, we need to show that the following inequality holds for any observation $\mathcal{H}\subseteq\text{Range}(C)$ and any pair of adjacent cost function sets $\mathcal{S}^{(1)}$ and $\mathcal{S}^{(2)}$,
\begin{align*}
P\{(\xi_{x},\xi_{y},&\mathbf{\varrho})\in\Psi|Z_{\mathcal{F}^{(1)}}(\xi_{x},\xi_{y},\mathbf{\varrho})\in\mathcal{H}\}\\
&\leq e^\epsilon P\{(\xi_{x},\xi_{y},\mathbf{\varrho})\in\Psi|Z_{\mathcal{F}^{(2)}}(\xi_{x},\xi_{y},\mathbf{\varrho})\in\mathcal{H}\},
\end{align*}
where $\mathcal{F}^{(l)}\!\!=\!\!\{\x(0),W,\mathcal{S}^{(l)}\!\}$, $l\!=\!1,2$, and $\Psi$ denotes the sample space. Then it is indispensable to guarantee $Z_{\mathcal{F}^{(1)}}(\xi_{x},\xi_{y},\mathbf{\varrho})=$ $Z_{\mathcal{F}^{(2)}}(\xi_{x},\xi_{y},\mathbf{\varrho})$, i.e.,
\begin{align*}
&C(x_{i,k}^{a-c,(1)},\varrho_k)=C(x_{i,k}^{a-c,(2)},\varrho_k),\addtag\label{eq:proofep1}\\
&C(y_{i,k}^{a-c,(1)},\varrho_k)=C(y_{i,k}^{a-c,(2)},\varrho_k),\addtag\label{eq:proofep2}
\end{align*}
for $\forall i\in\mathcal{V}$ and any $k\geq0$, where
\begin{align*}
&x_{i,k}^{a-c,(1)}=x_{i,k}^{a,(1)}-x_{i,k-1}^{c,(1)},\\
&y_{i,k}^{a-c,(1)}=y_{i,k}^{a,(1)}-y_{i,k-1}^{c,(1)}.
\end{align*}
Since $x_{i,0}^{c,(1)}=y_{i,0}^{c,(1)}=0,$ $l=1,2,$ from~\eqref{citerationx}--\eqref{citerationyc} and~\eqref{eq:proofep1}--\eqref{eq:proofep2}, we have
\begin{align*}
x_{i,k}^{c,(1)}=x_{i,k}^{c,(2)},
y_{i,k}^{c,(1)}=y_{i,k}^{c,(2)},~k=0,\dots,\infty.
\end{align*}
Then one obtains that 
\begin{align}\label{eq:pdfofcompressor}
\begin{aligned}
&f_c(x_{i,k}^{a-c,(1)},\varrho_k)=f_c(x_{i,k}^{a-c,(2)},\varrho_k),\\
&f_c(y_{i,k}^{a-c,(1)},\varrho_k)	=f_c(y_{i,k}^{a-c,(2)},\varrho_k),
\end{aligned}
\end{align}
if $x_{i,k}^{a,(1)}=x_{i,k}^{a,(2)}$ and $y_{i,k}^{a,(1)}=y_{i,k}^{a,(2)}$, for $\forall i\in\mathcal{V}$. Then due to the property of conditional probability, we have
\begin{align}\label{eq:ineqofprivacy}
\begin{aligned}
&\frac{P\{(\xi_{x},\xi_{y},\mathbf{\varrho})\in\Psi|Z_{\mathcal{F}^{(1)}}(\xi_{x},\xi_{y},\mathbf{\varrho})\in\mathcal{H}\}}{P\{(\xi_{x},\xi_{y},\mathbf{\varrho})\in\Psi|Z_{\mathcal{F}^{(2)}}(\xi_{x},\xi_{y},\mathbf{\varrho})\in\mathcal{H}\}}\\
&\leq\frac{P\{(\xi_{x},\xi_{y})\in\Psi|Z_{\mathcal{F}^{(1)}}(\xi_{x},\xi_{y})\in\mathcal{H}\}}{P\{(\xi_{x},\xi_{y})\in\Psi|Z_{\mathcal{F}^{(2)}}(\xi_{x},\xi_{y})\in\mathcal{H},E_1\}},
\end{aligned}
\end{align}	
where $E_1=\cup_{k=0}^\infty\{x_{i,k}^{a,(2)}=x_{i,k}^{a,(1)}, y_{i,k}^{a,(2)}=y_{i,k}^{a,(1)},\forall i\in\mathcal{V}\}$ is an event. We then analyze the right side of the inequality~\eqref{eq:ineqofprivacy}. Since event $E_1$ holds, one obtains that
\begin{align}\label{eq:privacyau1}
\xi_{x_{i},k}^{(1)}=\xi_{x_{i},k}^{(2)},~\xi_{y_{i},k}^{(1)}=\xi_{y_{i},k}^{(2)}~~\forall~k\in\mathbb{N},~\forall i\neq i_0.
\end{align}
From~\eqref{iterationx}--\eqref{iterationy}, the noises with respect to agent~$i_0$ should satisfy
\begin{align*}
    &\Delta \xi_{x_{i_0},k}=-\Delta x_{i_0,k},\addtag\label{eq:privacyau2}\\
    &\Delta \xi_{y_{i_0},k}=-\Delta y_{i_0,k},\addtag\label{eq:privacyau3}\\
    &\Delta x_{i_0,k+1}=-\eta\Delta y_{i_0,k},\addtag\label{eq:privacyau4}\\
    &\Delta y_{i_0,k+1}=\Delta f_{i_0,k+1}-\Delta f_{i_0,k},\addtag\label{eq:privacyau5}
\end{align*}
where $\Delta \xi_{x_{i_0},k}\triangleq\xi_{x_{i_0},k}^{(1)}-\xi_{x_{i_0},k}^{(2)},~\Delta \xi_{y_{i_0},k}\triangleq\xi_{y_{i_0},k}^{(1)}-\xi_{y_{i_0},k}^{(2)}$, $\Delta x_{i_0,k}\triangleq x_{i_0,k}^{(1)}-x_{i_0,k}^{(2)},~\Delta y_{i_0,k}\triangleq y_{i_0,k}^{(1)}-y_{i_0,k}^{(2)}$, and $\Delta f_{i_0,k}=\lf_{i_0}^{(1)}(x_{i_0,k}^{(1)})-\lf_{i_0}^{(2)}(x_{i_0,k}^{(2)})$. From~\eqref{eq:privacyau1}--\eqref{eq:privacyau5}, we know for any pair $(\xi_x^{(1)},\xi_y^{(1)})$, there exist a unique pair $(\xi_x^{(2)},\xi_y^{(2)})=(\xi_x^{(1)}+\Delta\xi_x,~\xi_y^{(1)}+\Delta\xi_y)$ such that $Z_{\mathcal{F}^{(1)}}(\xi_{x}^{(1)},\xi_{y}^{(1)})=Z_{\mathcal{F}^{(1)}}(\xi_{x}^{(2)},\xi_{y}^{(2)})$. Let $\mathcal{I}^{(l)}=\{(\xi_x^{(l)},\xi_y^{(l)})|Z_{\mathcal{F}^{(l)}}(\xi_{x}^{(l)},\xi_{y}^{(l)})\in\mathcal{H}\}$, $l=1,2$. Then we have
\begin{align*}
    &\frac{P\{(\xi_{x},\xi_{y})\in\Psi|Z_{\mathcal{F}^{(1)}}(\xi_{x},\xi_{y})\in\mathcal{H}\}}{P\{(\xi_{x},\xi_{y})\in\Psi|Z_{\mathcal{F}^{(2)}}(\xi_{x},\xi_{y})\in\mathcal{H},E_1\}}\\
    &=\frac{P\{(\xi_{x}^{(1)},\xi_{y}^{(1)})\in\mathcal{I}^{(1)}\}}{P\{(\xi_{x}^{(2)},\xi_{y}^{(2)})\in\mathcal{I}^{(2)}\}}\\
    &=\frac{\int_{\mathcal{I}^{(1)}} f_{\xi}\left(\xi_x^{(1)},\xi_y^{(1)}\right) d \xi_x^{(1)}\xi_y^{(1)}}{\int_{\mathcal{I}^{(1)}} f_{\xi}\left(\xi_x^{(1)}+\Delta \xi_x,\xi_y^{(1)}+\Delta \xi_y\right) d \xi_x^{(1)}\xi_y^{(1)}},\addtag\label{eq:privacyau6}
\end{align*}
where 
\begin{align*}
    &f_{\xi}\left(\xi_x^{(l)},\xi_y^{(l)}\right)=\prod_{k=0}^K \prod_{i=1}^n \prod_{r=1}^d f_L\left([\xi_{x_i, k}^{(l)}]_r, s_{\xi_{x_i}}q_i^k\right)\\
    &~~~~~~~~~~f_L\left([\xi_{y_i, k}^{(l)}]_r, s_{\xi_{y_i}}q_i^k\right).
\end{align*}
Then~\eqref{eq:privacyau6} can be rewritten as
\begin{align*}
&\frac{f_{\xi}\left(\xi_x^{(1)},\xi_y^{(1)}\right)}{f_{\xi}\left(\xi_x^{(1)}+\Delta \xi_x,\xi_y^{(1)}+\Delta \xi_y\right)}\\
&=\prod_{k=0}^K \prod_{i=1}^n \prod_{r=1}^d \frac{f_L\left([\xi_{x_i, k}^{(1)}]_r, s_{\xi_{x_i}}q_i^k\right)}{f_L\left([\xi_{x_i, k}^{(1)}+\Delta\xi_{x_i, k}]_r, s_{\xi_{x_i}}q_i^k\right)}\\
&~~~~~~~\frac{f_L\left([\xi_{y_i, k}^{(1)}]_r, s_{\xi_{y_i}}q_i^k\right)}{f_L\left([\xi_{y_i, k}^{(1)}+\Delta\xi_{y_i, k}]_r, s_{\xi_{y_i}}q_i^k\right)}\\
&\leq \exp\left(\sum_{k=0}^K\frac{\Vert\Delta \xi_{x_{i_0}}\Vert_1}{s_{\xi_{x_{i_0}}}q_{i_0}^k}+\frac{\Vert\Delta \xi_{y_{i_0}}\Vert_1}{s_{\xi_{y_{i_0}}}q_{i_0}^k}\right).
\end{align*}
From~\eqref{eq:privacyau2}--\eqref{eq:privacyau5}, we have
\begin{align*}
    \Vert\Delta \xi_{y_{i_0}}\Vert_1\leq 4\sqrt{d}M,~\Vert\Delta \xi_{x_{i_0}}\Vert_1\leq 4\sqrt{d\eta} M,
\end{align*}
Then we complete the proof.


\section{The proof of Theorem~\ref{theo:convergence1}}\label{app-convergence1}
\subsection{Supporting Lemmas}
\begin{lemma}
~(Lemma 2 in~\cite{yi2022communication}) Suppose Assumption~\ref{as:strongconnected} holds, let $L$ be the Laplacian matrix of the graph $G$ and $K_n=\mathbf{I}_n-\frac{1}{n}\mathbf{1}_n\mathbf{1}_n^\top$. Then $L$ and $K_n$ are positive semi-definite, $L\leq\bar{\lambda}_L\mathbf{I}_n$, $\bar{\lambda}_{K_n}=1$, 
\begin{align}
&K_nL=LK_n=L,\label{eq:propertyofk}\\
&0\leq\underline{\lambda}_LK_n\leq L\leq\bar{\lambda}_LK_n.\label{eq:propertyofk1}
\end{align}
Moreover, there exists an orthogonal matrix $[r~R]\in\R^{n\times n}$ with $r=\frac{1}{\sqrt{n}}\mathbf{1}_n$ and $R\in\R^{n\times(n-1)}$ such that
\begin{align}
	&PL=LP=K_n,\label{eq:propertyofp}\\
	&\bar{\lambda}_L^{-1}\mathbf{I}_n\leq P\leq\underline{\lambda}_L^{-1}\mathbf{I}_n,\label{eq:propertyofp1}
\end{align}
where 
\begin{align*}
	P=\begin{bmatrix}
		r&R
	\end{bmatrix}
	\begin{bmatrix}
		\lambda_n^{-1}&0\\
		0&\Lambda_1^{-1}
	\end{bmatrix}
	\begin{bmatrix}
		r^\top\\
		R^\top
	\end{bmatrix},
\end{align*}
with $\Lambda_1=\text{diag}([\lambda_2,\dots,\lambda_n])$ and $0\leq\lambda_2\leq\cdots\leq\lambda_n$ being the nonzero eigenvalues of $L$.
\end{lemma}

Denote $\tilde{f}(\x_{k})=\sum_{i=1}^nf_i(x_{i,k}),~\bl=L\otimes \bi_d$, $\bp=P\otimes\bi_d$, $\g_k=\nabla\tilde{f}(\x_{k})$, $\bg_k=\bh\g_k$, $\g_k^b=\nabla\tilde{f}(\bx_{k})$, $\bg_k^b=\bh\g_k^b=\mathbf{1}_n\otimes\nabla f(\bar{x}_k)$.

Before proving Theorem~\ref{theo:convergence2}, we provide the inequality regarding with two state values by using the following lemma.
\begin{lemma}\label{lemma:lyapunov}
Suppose Assumptions~\ref{as:smooth}--\ref{as:finite} and \ref{as:strongconnected}--\ref{as:compressor} hold. Under PPDC, if $\alpha_x,\alpha_y\in(0,\frac{1}{r})$, we have
\begin{align*}
	\widetilde{V} _{k+1}&\leq\widetilde{V} _{k}-\Vert\x_k\Vert^2_{(\kappa_1-\kappa_2\eta)\eta\bk}-\Vert\vv_k+\frac{1}{\omega}\g_k^b\Vert^2_{(\kappa_3-\kappa_4\eta)\eta\bp}\\
	&~~~-(\kappa_5-\kappa_6\eta)\eta \Vert\bg_k\Vert^2-(\kappa_7-\kappa_8\eta-\kappa_9\eta^2)\Vert\x_{k}\\
	&~~~+\nxk-\x^c_{k}\Vert^2-\frac{\eta}{4}\Vert\bg_k^b\Vert^2+\kappa_{10}\Vert\nxk\Vert^2+\kappa_{11}\Vert\nvk\Vert^2\\
	&~~~+\kappa_{12}\Vert\sum_{t=0}^{k}\bar{\xi}_{v,t}\Vert^2,\addtag\label{eq:lyapunov1}
\end{align*}
where
\begin{align*}
	&\widetilde{V} _k=\frac{1}{2}\Vert\x_{k}\Vert_\bk^2+\frac{1}{2}\Vert\vv_{k}+\frac{1}{\omega}\g_{k}^b\Vert^2_{\bp+\frac{\gamma}{\omega}\bp}+\x_{k}^\top\bk\bp(\vv_{k}+\frac{1}{\omega}\g_{k}^b)\\
	&~~~~~~+\Vert\x_{k}+\nxk-\x^c_{k}\Vert^2+n(f(\bar{x}_{k})-f^*),\\
	&\kappa_1=\frac{\gamma\underline{\lambda}_L}{2}-\frac{1}{4}(7+9L_f^2+13\omega),\\
	&\kappa_2=\frac{5L_f^2}{2}+\frac{6(2+\alpha_xr\varphi)}{\alpha_xr\varphi}(\gamma^2\bar{\lambda}_L^2+L_f^2)+6\gamma^2\bar{\lambda}_L^2\\
	&~~~~~~+\frac{9\omega^2+3\gamma^2}{2}\bar{\lambda}_L+3\omega^2+\frac{3}{2},\\
	&\kappa_3=\frac{\omega}{4}-\frac{3}{\underline{\lambda}_L}-(\frac{1}{2}+\frac{1}{2\omega})(1+\frac{\gamma}{\omega}),\\
	&\kappa_4=\frac{1}{\underline{\lambda}_L}+\frac{\omega^2\bar{\lambda}_L}{2}+2\omega^2\bar{\lambda}_L+6(1+\frac{2}{\alpha_xr\varphi})\omega^2\underline{\lambda}_L,\\
	&\kappa_5=\frac{1}{4}-(\frac{L_f^2}{\omega\underline{\lambda}_L}+\frac{\gamma L_f^2}{\omega^2\underline{\lambda}_L}+\frac{L_f^2}{\omega^2\underline{\lambda}_L^2}),\\
	&\kappa_6=\frac{2L_f^2}{\omega^2\underline{\lambda}_L}+\frac{2L_f^2}{\omega^3\underline{\lambda}_L}+\frac{L_f^2}{\omega^2\underline{\lambda}_L^2}+\frac{3L_f^2}{2}+L_f,\\
	&\kappa_7=\frac{\alpha_xr\varphi}{2}(1+\alpha_xr\varphi),\\
	&\kappa_8=\frac{1}{2}(\gamma+4\omega)\bar{\lambda}_L+3\omega r_0\\
	&\kappa_9=6\gamma^2\bar{\lambda}_L^2r_0(1+\frac{2}{\alpha_xr\varphi})+ r_0(6\gamma^2\bar{\lambda}_L^2+\frac{9\omega^2+3\gamma^2}{2}\bar{\lambda}_L\\
	&~~~~~~+3\omega^2+\frac{3}{2}),\\
	&\kappa_{10}=4(\frac{L_f^2}{\omega^2\underline{\lambda}_L}+\frac{\gamma L_f^2}{\omega^3\underline{\lambda}_L}+\frac{L_f^2}{2\eta\omega\underline{\lambda}_L}+\frac{\gamma L_f^2}{2\eta\omega^2\underline{\lambda}_L}\\
	&~~~~~~+\frac{(1+2\eta)L_f^2}{2\eta\omega^2\underline{\lambda}_L^2}+\frac{3L_f^2}{4}+\frac{1}{2\eta}+\frac{L_f}{2})+\frac{\bar{\lambda}_{\bi_{nd}+\eta\gamma\bl}^2}{\eta}\\
	&~~~~~~+10+(\frac{1}{2}+3\omega+2\omega\bar{\lambda}_L)\eta+(\frac{1}{\eta\omega}+\frac{1}{2})\frac{1}{\underline{\lambda}_L}+\frac{12}{\alpha_x r\varphi}\\
	&~~~~~~+6\bar{\lambda}^2_{\eta\gamma\bl-\bi}(1+\frac{2}{\alpha_x r\varphi})+\eta^2(6\gamma^2\bar{\lambda}_L^2+\frac{9\omega^2+3\gamma^2}{2}\bar{\lambda}_L\\
	&~~~~~~+3\omega^2+\frac{3}{2}),\\
	&\kappa_{11}=\underline{\lambda}_L^{-1}(\frac{1}{2\eta}+\frac{\gamma}{2\eta\omega}+\frac{5}{2}+\frac{3\gamma}{2\omega}+\eta\omega+\frac{\underline{\lambda}_L^{-1}}{\eta}),\\
	&\kappa_{12}=4\eta^2\omega^2(\frac{L_f^2}{\omega^2\underline{\lambda}_L}+\frac{\gamma L_f^2}{\omega^3\underline{\lambda}_L}+\frac{L_f^2}{2\eta\omega\underline{\lambda}_L}+\frac{\gamma L_f^2}{2\eta\omega^2\underline{\lambda}_L}\\
	&~~~~~~+\frac{(1+2\eta)L_f^2}{2\eta\omega^2\underline{\lambda}_L^2}+\frac{3L_f^2}{4}+\frac{1}{2\eta}+\frac{L_f}{2})+\frac{\eta}{4}.
\end{align*}
\end{lemma}
\subsection{The proof of Lemma~\ref{lemma:lyapunov}}

\textbf{(i)} For simplicity of the proof, we first provide some useful properties. The update equations~\eqref{eq:iterationxp2}, \eqref{eq:iterationv2}, \eqref{citerationx}, and~\eqref{citerationy} can be rewritten as the following compact form
\begin{align*}
	&~\x_{k+1}=\x_{k}+\nxk-\eta(\gamma \bl\hx_{k}+\omega  \vv_{k}+\nabla\tilde{f}(\x_{k})),\addtag\label{eq:iterationxp3}\\
	&~\vv_{k+1}=\vv_{k}+\nvk+\eta\omega \bl\hx_{k},\addtag\label{eq:iterationv3}\\
	&~\hx_{k}=\x_{k}^c+C(\x_{k}^a-\x_{k}^c),\addtag\label{citerationx1}\\
	&~\x_{k+1}^c=(1-\alpha_x)\x_{k}^c+\alpha_x\hx_{k},\addtag\label{citerationxc1}
	\end{align*}
   From~\eqref{eq:iterationv3}, the propoerty of Laplacian matrix, and the fact that $\sum_{i=1}^n v_{i,0}=\mathbf{0}_d$, we have 
   \begin{align}\label{eq:propertyofv}
	\bv_{k+1}=\sum_{t=0}^{k}\bar{\xi}_{v,t}.
   \end{align}
Then from~\eqref{eq:iterationxp3} and~\eqref{eq:propertyofv}, one obtains that
\begin{align}\label{eq:propertyofbx}
	\bx_{k+1}=\bx_k+\nbxk-\eta\bg_k-\eta\omega\sum_{t=0}^{k}\bar{\xi}_{v,t}.
\end{align}
Furthermore, we have following useful equations
\begin{align*}
	&\Vert\g^b_k-\g_k\Vert^2\leq L_f^2\Vert\bx_k-\x_k\Vert^2\leq L_f^2\Vert\x_k\Vert^2_\bk,\addtag\label{eq:propertyofbg1}\\
	&\Vert\bg^b_k-\bg_k\Vert^2=\Vert\bh(\g^b_k-\g_k)\Vert^2\leq L_f^2\Vert\x_k\Vert^2_\bk,\addtag\label{eq:propertyofbg2}\\
	&\Vert\g^b_{k+1}-\g^b_k\Vert^2\leq L_f^2\Vert\bx_{k+1}-\bx_k\Vert^2\\
	&~~~~~~~~~~~~~~~~~\leq L_f^2\Vert\nbxk-\eta\bg_k-\eta\omega\sum_{t=0}^{k}\bar{\xi}_{v,t}\Vert^2\addtag\label{eq:propertyofbg3},
\end{align*}
where the first inequality comes from Assumption~\ref{as:smooth} and $\lb_\bk=1$; the second inequality comes from~\eqref{eq:propertyofbg1} and $\lb_\bh=1$; the last inequality comes from Assumption~\ref{as:smooth} and~\eqref{eq:propertyofbx}.

\textbf{(ii)} The proof of Lemma~\ref{lemma:lyapunov}.
We first provide the upper bound of~$\frac{1}{2}\Vert\x_{k+1}\Vert_\bk^2$
\begin{align*}
\frac{1}{2}\Vert\x_{k+1}\Vert_\bk^2&=\frac{1}{2}\Vert\x_{k}+\nxk-\eta(\gamma \bl\hx_{k}+\omega  \vv_{k}+\g_k)\Vert_\bk^2\\
&=\frac{1}{2}\Vert\x_{k}+\nxk-\eta\gamma \bl\hx_{k}\Vert_\bk^2\\
&~~~-\eta\omega (\x_{k}+\nxk-\eta\gamma \bl\hx_{k})^\top\bk(\vv_k+\frac{1}{\omega}\g_k)\\
&~~~+\Vert\vv_k+\frac{1}{\omega}\g_k\Vert^2_{\frac{\eta^2\omega^2}{2}\bk}\\
&\leq\frac{1}{2}\Vert\x_{k}\Vert_\bk^2+\frac{1}{2}\Vert\nxk-\eta\gamma \bl\hx_{k}\Vert_\bk^2\\
&~~~+\x_k^\top\bk((\bi_{nd}+\eta\gamma\bl)\nxk-\eta\gamma \bl(\hx_{k}-\nxk))\\
&~~~-\eta\omega (\x_{k}+\nxk-\eta\gamma \bl\hx_{k})^\top\bk\left(\vv_k+\frac{1}{\omega}\g_k^b\right)\\
&~~~+\frac{\eta}{2}\Vert\x_k\Vert_{\bk}^2+\frac{\eta}{2}\Vert\g_k-\g_k^b\Vert^2+\Vert\hx_k\Vert^2_{\frac{\eta^2\gamma^2}{2}\bl^2}\\
&~~~+\frac{\eta^2}{2}\Vert\g_k-\g_k^b\Vert^2+\frac{\eta}{2}\Vert\nxk\Vert_{\bk}^2+\frac{\eta}{2}\Vert\g_k-\g_k^b\Vert^2\\
&~~~+\Vert\vv_k+\frac{1}{\omega}\g_k^b+\frac{1}{\omega}\g_k-\frac{1}{\omega}\g_k^b\Vert^2_{\frac{\eta^2\omega^2}{2}\bk}\\
&\leq\frac{1}{2}\Vert\x_{k}\Vert_\bk^2-\eta\gamma\x_k^\top\bl(\hx_k+\nxk)+\frac{\eta}{4}\Vert\x_{k}\Vert_\bk^2\\
&~~~+\frac{1}{\eta}\Vert(\bi_{nd}+\eta\gamma\bl)\nxk\Vert_\bk^2+\Vert\nxk\Vert_\bk^2\\
&~~~+\Vert\hx_k\Vert^2_{\eta^2\gamma^2\bl^2}-\eta\omega (\x_{k})^\top\bk\left(\vv_k+\frac{1}{\omega}\g_k^b\right)\\
&~~~+\frac{1}{2}\Vert\nxk\Vert^2_{\bk}+\frac{\eta^2\omega^2}{2}\Vert\vv_k+\frac{1}{\omega}\g_k^b\Vert^2_{\bk}\\
&~~~+\Vert\hx_k\Vert^2_{\frac{\eta^2\gamma^2}{2}\bl^2}+\frac{\eta^2\omega^2}{2}\Vert\vv_k+\frac{1}{\omega}\g_k^b\Vert^2_{\bk}\\
&~~~+\frac{\eta}{2}\Vert\x_k\Vert_{\bk}^2+\eta \Vert\g_k-\g_k^b\Vert^2+\Vert\hx_k\Vert^2_{\frac{\eta^2\gamma^2}{2}\bl^2}\\
&~~~+\frac{3\eta^2}{2}\Vert\g_k-\g_k^b\Vert^2+\frac{\eta}{2}\Vert\nxk\Vert_{\bk}^2\\
&~~~+\eta^2\omega^2\Vert\vv_k+\frac{1}{\omega}\g_k^b\Vert^2_{\bk}\\
&\leq\frac{1}{2}\Vert\x_{k}\Vert_\bk^2-\Vert\x_k\Vert_{\eta\gamma\bl}^2+\Vert\hx_k+\nxk-\x_k\Vert_{\frac{\eta\gamma\bl}{2}}^2\\
&~~~+\Vert\x_k\Vert_{\frac{\eta\gamma\bl}{2}}^2+\frac{3\eta}{4}\Vert\x_{k}\Vert_\bk^2\\
&~~~+\Vert\nxk\Vert_{(\frac{\bar{\lambda}_{\bi_{nd}+\eta\gamma\bl}^2}{\eta}+\frac{3}{2}+\frac{\eta}{2})\bk}^2+\Vert\hx_k\Vert^2_{2\eta^2\gamma^2\bl^2}\\
&~~~+\eta(\frac{3\eta}{2}+1)\Vert\g_k-\g_k^b\Vert^2\!+\!\Vert\vv_k+\frac{1}{\omega}\g_k^b\Vert^2_{2\eta^2\omega^2\bk}\\
&~~~-\!\eta\omega (\hx_k\!+\!\x_{k}\!-\!\hx_k\!+\!\nxk\!-\!\nxk)^\top\bk(\vv_k\!+\!\frac{1}{\omega}\g_k^b)\\
&\leq\frac{1}{2}\Vert\x_{k}\Vert_\bk^2-\Vert\x_k\Vert_{\frac{\eta\gamma\bl}{2}-(\frac{3\eta}{4}+\eta(\frac{3\eta}{2}+1)L_f^2)\bk}^2\\
&~~~+\Vert\hx_k\Vert^2_{2\eta^2\gamma^2\bl^2}\\
&~~~+\Vert\hx_k+\nxk-\x_k\Vert_{\frac{\eta}{2}(\gamma\bl+4\omega\bar{\lambda}_L\bk) }^2\\
&~~~+\Vert\nxk\Vert_{(\frac{\bar{\lambda}_{\bi_{nd}+\eta\gamma\bl}^2}{\eta}+\frac{3}{2}+\frac{\eta}{2}+2\eta\omega\bar{\lambda}_L)\bk}^2\\
&~~~+\Vert\vv_k+\frac{1}{\omega}\g_k^b\Vert^2_{2\eta^2\omega^2+\frac{\eta\omega\bar{\lambda}_L^{-1}}{4}}\\
&~~~-\eta\omega (\hx_k)^\top\bk\left(\vv_k\!+\!\frac{1}{\omega}\g_k^b\right),\addtag\label{eq:upperboundofx1}
\end{align*}
where the first and second equalities comes from~\eqref{eq:iterationxp3}; the first, second, and third inequalities comes from~\eqref{eq:inequality1} and~\eqref{eq:propertyofk}; the last inequality comes from~\eqref{eq:inequality1},~\eqref{eq:propertyofbg1}, and $\lb_\bk=1$. 

We then provide the upper bound of $\frac{1}{2}\Vert\vv_{k+1}+\frac{1}{\omega}\g_{k+1}^b\Vert^2_{\bp+\frac{\gamma}{\omega}\bp}$.
\begin{align*}
	\frac{1}{2}\Vert&\vv_{k+1}+\frac{1}{\omega}\g_{k+1}^b\Vert^2_{\bp+\frac{\gamma}{\omega}\bp}\\
	&=\frac{1}{2}\Vert\vv_k\!+\!\frac{1}{\omega}\g_k^b\!+\!\nvk\!+\!\eta\omega\bl\hx_k\!+\!\frac{1}{\omega}(\g_{k+1}^b-\g_k^b)\Vert^2_{\bp+\frac{\gamma}{\omega}\bp}\\
	&\leq \frac{1}{2}\Vert\vv_k+\frac{1}{\omega}\g_k^b\Vert^2_{\bp+\frac{\gamma}{\omega}\bp}+\eta(\gamma+\omega)\hx_k^\top\bk\left(\vv_k+\frac{1}{\omega}\g_k^b\right)\\
	&~~~+\nvk^\top(\bp+\frac{\gamma}{\omega}\bp)\left(\vv_k+\frac{1}{\omega}\g_k^b\right)\\
	&~~~+\frac{1}{\omega}(\g_{k+1}^b-\g_k^b)^\top(\bp+\frac{\gamma}{\omega}\bp)\left(\vv_k+\frac{1}{\omega}\g_k^b\right)\\
	&~~~+\eta\hx_k^\top(\bk+\frac{\gamma}{\omega}\bk)(\g_{k+1}^b-\g_k^b)+\frac{3}{2}\Vert\nvk\Vert^2_{\bp+\frac{\gamma}{\omega}\bp}\\
	&~~~+\Vert\hx_k\Vert^2_{\eta^2\omega(\gamma+\omega)\bl}+\frac{1}{\omega^2}\Vert\g_{k+1}^b-\g_k^b\Vert^2_{\bp+\frac{\gamma}{\omega}\bp}\\
	&\leq \frac{1}{2}\Vert\vv_k+\frac{1}{\omega}\g_k^b\Vert^2_{\bp+\frac{\gamma}{\omega}\bp}+\eta(\gamma+\omega)\hx_k^\top\bk\left(\vv_k+\frac{1}{\omega}\g_k^b\right)\\
	&~~~+\frac{1}{2\eta}\Vert\nvk\Vert^2_{\bp+\frac{\gamma}{\omega}\bp}+\frac{\eta}{2}\Vert\vv_k+\frac{1}{\omega}\g_k^b\Vert^2_{\bp+\frac{\gamma}{\omega}\bp}\\
	&~~~+\frac{1}{2\eta\omega}\Vert\g_{k+1}^b-\g_k^b\Vert^2_{\bp+\frac{\gamma}{\omega}\bp}+\frac{\eta}{2\omega}\Vert\vv_k+\frac{1}{\omega}\g_k^b\Vert^2_{\bp+\frac{\gamma}{\omega}\bp}\\
	&~~~+\Vert\hx_k\Vert_{\frac{\eta^2}{2}\bk}^2\!+\!\frac{1}{2}\Vert\g_{k+1}^b-\g_k^b\Vert^2\!+\!\frac{\eta\gamma}{\omega}\hx_k^\top\bk(\g_{k+1}^b-\g_k^b)\\
	&~~~+\frac{3}{2}\Vert\nvk\Vert^2_{\bp+\frac{\gamma}{\omega}\bp}+\Vert\hx_k\Vert^2_{\eta^2\omega(\gamma+\omega)\bl}\\
	&~~~+\frac{1}{\omega^2}\Vert\g_{k+1}^b-\g_k^b\Vert^2_{\bp+\frac{\gamma}{\omega}\bp}\\
	&=\frac{1}{2}\Vert\vv_k+\frac{1}{\omega}\g_k^b\Vert^2_{\bp+\frac{\gamma}{\omega}\bp}+\eta(\gamma+\omega)\hx_k^\top\bk\left(\vv_k+\frac{1}{\omega}\g_k^b\right)\\
	&~~~+\Vert\hx_k\Vert^2_{\eta^2\omega(\gamma+\omega)\bl+\frac{\eta^2}{2}\bk}+\Vert\vv_k+\frac{1}{\omega}\g_k^b\Vert^2_{(\frac{\eta}{2}+\frac{\eta}{2\omega})(\bp+\frac{\gamma}{\omega}\bp)}\\
	&~~~+\Vert\g_{k+1}^b-\g_k^b\Vert^2_{(\frac{1}{\omega^2}+\frac{1}{2\eta\omega})(\bp+\frac{\gamma}{\omega}\bp)+\frac{1}{2}\bi}\\
	&~~~+\frac{\eta\gamma}{\omega}\hx_k^\top\bk(\g_{k+1}^b-\g_k^b)+(\frac{1}{2\eta}+\frac{3}{2})\Vert\nvk\Vert^2_{\bp+\frac{\gamma}{\omega}\bp}\\
	&\leq\frac{1}{2}\Vert\vv_k+\frac{1}{\omega}\g_k^b\Vert^2_{\bp+\frac{\gamma}{\omega}\bp}+\eta(\gamma+\omega)\hx_k^\top\bk\left(\vv_k+\frac{1}{\omega}\g_k^b\right)\\
	&~~~+\Vert\hx_k\Vert^2_{\eta^2\omega(\gamma+\omega)\bl+\frac{\eta^2}{2}\bk}+\Vert\vv_k+\frac{1}{\omega}\g_k^b\Vert^2_{(\frac{\eta}{2}+\frac{\eta}{2\omega})(\bp+\frac{\gamma}{\omega}\bp)}\\
	&~~~+2\tilde{b}_1\eta^2L_f^2\Vert\bg_k\Vert^2+4\tilde{b}_1L_f^2\Vert\nbxk\Vert^2\\
	&~~~+4\tilde{b}_1\eta^2\omega^2L_f^2\Vert\sum_{t=0}^k\mathbf{\bar{\xi}}_{v,t}\Vert^2+\frac{\eta\gamma}{\omega}\hx_k^\top\bk(\g_{k+1}^b-\g_k^b)\\
	&~~~+(\frac{1}{2\eta}+\frac{3}{2})\Vert\nvk\Vert^2_{\bp+\frac{\gamma}{\omega}\bp},\addtag\label{eq:upperboundofv1}
\end{align*}
where $\tilde{b}_1=(\frac{1}{\omega^2}+\frac{1}{2\eta\omega})(1+\frac{\gamma}{\omega})\frac{1}{\underline{\lambda}_L}+\frac{1}{2}$; the first equality comes from~\eqref{eq:iterationv3}; the first inequality comes from~\eqref{eq:inequality2} and~\eqref{eq:propertyofp}; the second inequality holds due to~\eqref{eq:inequality1}, \eqref{eq:inequality2}, and~\eqref{eq:propertyofp}; the last inequality comes from~\eqref{eq:inequality2},~\eqref{eq:propertyofp1}, and~\eqref{eq:propertyofbg3}. 

We then provide the upper bound of~$\x_{k+1}^\top\bk\bp(\vv_{k+1}+\frac{1}{\omega}\g_{k+1}^b)$
\begin{align*}	\x_{k+1}^\top&\bk\bp(\vv_{k+1}+\frac{1}{\omega}\g_{k+1}^b)\\
	&=(\x_k+\nxk-\eta(\gamma\bl\hx_k+\omega\vv_k+\g_k^b+\g_k-\g_k^b))^\top\\
	&~~~\bk\bp(\vv_k+\nvk+\frac{1}{\omega}\g_k^b+\eta\omega\bl\hx_k+\frac{1}{\omega}(\g_{k+1}^b-\g_k^b))\\
	&=(\x_k-\eta(\gamma\bl\hx_k+\omega\vv_k+\g_k^b+\g_k-\g_k^b))^\top\bk\bp(\vv_k\\
	&~~~+\frac{1}{\omega}\g_k^b+\eta\omega\bl\hx_k+\frac{1}{\omega}(\g_{k+1}^b-\g_k^b))\\
	&~~~+\nxk^\top\bk\bp(\vv_k+\nvk+\frac{1}{\omega}\g_k^b+\eta\omega\bl\hx_k\\
	&~~~+\frac{1}{\omega}(\g_{k+1}^b-\g_k^b))+\nvk^\top\bk\bp(\x_k-\eta(\gamma\bl\hx_k+\omega\vv_k\\
	&~~~+\g_k^b+\g_k-\g_k^b))\\
	&\leq(\x_k^\top\bk\bp-\eta(\gamma+\eta\omega^2)\hx_k^\top\bk)(\vv_k+\frac{1}{\omega}\g_k^b)\\
	&~~~+\eta\omega\x_k^\top\bk\hx_k-\Vert\hx_k\Vert^2_{\eta^2\gamma\omega\bl}+(\x_k^\top\bk\bp-\eta\gamma\hx_k^\top\bk)(\\
	&~~~\g_{k+1}^b-\g_k^b)\!-\eta(\omega\vv_k\!+\!\g_k^b+\!\g_k\!-\!\g_k^b\!-\!\sum_{t=0}^{k}\bar{\xi}_{v,t}-\bg_k)^\top\\
	&~~~\bp(\vv_k+\frac{1}{\omega}\g_k^b)-\eta(\vv_k+\frac{1}{\omega}\g_k^b)^\top\bp\bk(\g_{k+1}^b-\g_k^b)\\
	&~~~-\eta(\g_k-\g_k^b)^\top(\eta\omega\bk\hx_k+\frac{1}{\omega}\bk\bp(\g_{k+1}^b-\g_k^b))\\
	&~~~+\nxk^\top\bk\bp(\vv_k+\nvk+\frac{1}{\omega}\g_k^b+\eta\omega\bl\hx_k\\
	&~~~+\frac{1}{\omega}(\g_{k+1}^b-\g_k^b))+\nvk^\top\bp\bk(\x_k-\eta(\gamma\bl\hx_k+\omega\vv_k\\
	&~~~+\g_k^b+\g_k-\g_k^b))\\
	&\leq(\x_k^\top\bk\bp-\eta\gamma\hx_k^\top\bk)(\vv_k+\frac{1}{\omega}\g_k^b)+\Vert\hx_k\Vert_{\frac{\eta^2\omega^2}{2}\bk}^2\\
	&~~~+\Vert\vv_k+\frac{1}{\omega}\g_k^b\Vert_{\frac{\eta^2\omega^2}{2}\bk}^2+\Vert\x_k\Vert_{\frac{\eta\omega}{4}\bk}^2+\Vert\hx_k\Vert_{\eta\omega(\bk-\eta\gamma\bl)}^2\\
	&~~~+\Vert\x_k\Vert_{\frac{\eta}{2}\bk}^2+\Vert\g_{k+1}^b-\g_k^b\Vert^2_{\frac{1}{2\eta\omega^2}\bp^2}\\
	&~~~-\frac{\eta\gamma}{\omega}\hx_k^\top\bk(\g_{k+1}^b-\g_k^b)-\Vert\vv_k+\frac{1}{\omega}\g_k^b\Vert_{\eta\omega\bp}^2\\
	&~~~+\frac{\eta}{4}\Vert\g_k-\g_k^b\Vert^2+\frac{\eta}{4}\Vert\sum_{t=0}^{k}\bar{\xi}_{v,t}\Vert^2+\frac{\eta}{4}\Vert\bg_k\Vert^2\\
	&~~~+\Vert\vv_k+\frac{1}{\omega}\g_k^b\Vert_{3\eta\bp^2}^2+\Vert\vv_k+\frac{1}{\omega}\g_k^b\Vert_{\eta^2\bp^2}^2\\
	&~~~+\frac{1}{4}\Vert\g_{k+1}^b-\g_k^b\Vert^2+\frac{\eta^2}{2}\Vert\g_{k}-\g_k^b\Vert^2+\Vert\hx_k\Vert_{\frac{\eta^2\omega^2}{2}\bk}^2\\
	&~~~+\frac{\eta^2}{2}\Vert\g_{k}-\g_k^b\Vert^2+\Vert\g_{k+1}^b-\g_k^b\Vert_{\frac{1}{2\omega^2}\bp^2}^2\\
	&~~~+\frac{\eta\omega}{4}\Vert\vv_k+\frac{1}{\omega}\g_k^b\Vert_{\bp}^2+\frac{1}{\eta\omega}\Vert\nxk\Vert_{\bp}^2+\Vert\hx_k\Vert_{\frac{\eta^2\omega^2}{2}\bl}^2\\
	&~~~+\Vert\nxk\Vert_{\frac{1}{2}\bk\bp}^2+\Vert\g_{k+1}^b-\g_k^b\Vert^2_{\frac{1}{2\omega^2}\bk\bp^2}+\Vert\nxk\Vert_{\frac{1}{2}\bk}^2\\
	&~~~+\Vert\nvk\Vert_{\frac{1}{2}\bk\bp}^2+\Vert\nxk\Vert_{\frac{1}{2}\bk\bp}^2+\frac{\eta}{2}\Vert\x_k\Vert_{\bk}^2\\
	&~~~+\frac{1}{2\eta}\Vert\nvk\Vert_{\bp^2\bk}^2+\Vert\hx_k\Vert_{\frac{\eta^2\gamma^2}{2}\bl}^2+\Vert\nvk\Vert_{\frac{1}{2}\bp\bk}^2\\
	&~~~+\Vert\vv_k+\frac{1}{\omega}\g_k^b\Vert_{\frac{\eta\omega}{4}\bp\bk}^2+\Vert\nvk\Vert_{\eta\omega\bp\bk}^2+\frac{\eta}{2}\Vert\g_{k}-\g_k^b\Vert^2\\
	&~~~+\Vert\nvk\Vert_{\frac{1}{2\eta}\bp^2\bk^2}^2\\
	&\leq \x_k^\top\bk\bp(\vv_k+\frac{1}{\omega}\g_k^b)-\eta\gamma\hx_k^\top\bk(\vv_k+\frac{1}{\omega}\g_k^b)\\
	&~~~+\Vert\hx\Vert^2_{\eta\omega\bk+\eta^2(\omega^2\bk-\omega\gamma\bl+\frac{\omega^2+\gamma^2}{2}\bl)}+\Vert\x\Vert^2_{\frac{\eta(\omega+4)}{4}\bk}\\
	&~~~+\eta(\frac{3}{4}+\eta)\Vert\g_{k}-\g_k^b\Vert^2+\Vert\g_{k+1}^b-\g_k^b\Vert^2_{\frac{1+2\eta}{2\eta\omega^2}\bp^2+\frac{1}{4}\bi_{nd}}\\
	&~~~-\frac{\eta\gamma}{\omega}\hx_k^\top\bk(\g_{k+1}^b-\g_k^b)+\frac{\eta}{4}\Vert\bg_k\Vert^2\\
	&~~~-\Vert\vv_k+\frac{1}{\omega}\g_k^b\Vert^2_{\frac{\eta\omega}{2}\bp-3\eta\bp^2-\eta^2\bp^2-\frac{\eta^2\omega^2}{2}\bk}\\
	&~~~+\Vert\nxk\Vert^2_{(\frac{1}{\eta\omega}+\frac{1}{2})\bp+\frac{1}{2}\bk}+\Vert\nvk\Vert^2_{\frac{1}{\eta}\bp^2+\bp+\eta\omega\bp}\\
	&~~~+\frac{\eta}{4}\Vert\sum_{t=0}^{k}\bar{\xi}_{v,t}\Vert^2\\
	&\leq \x_k^\top\bk\bp(\vv_k+\frac{1}{\omega}\g_k^b)-\eta\gamma\hx_k^\top\bk(\vv_k+\frac{1}{\omega}\g_k^b)\\
	&~~~+\Vert\hx\Vert^2_{\eta\omega\bk+\eta^2(\omega^2\bk-\omega\gamma\bl+\frac{\omega^2+\gamma^2}{2}\bl)}\\
	&~~~+\Vert\x\Vert^2_{\frac{\eta(\omega+4)}{4}\bk+\eta(\frac{3}{4}+\eta)L_f^2\bk}+(2\tilde{b}_2 \eta^2L_f^2+\frac{\eta}{4})\Vert\bg_k\Vert^2\\
	&~~~-\frac{\eta\gamma}{\omega}\hx_k^\top\bk(\g_{k+1}^b-\g_k^b)\\
	&~~~-\Vert\vv_k+\frac{1}{\omega}\g_k^b\Vert^2_{\eta(\frac{\omega}{2}-\frac{3}{\underline{\lambda}_L})\bp-\eta^2(\frac{1}{\underline{\lambda}_L}+\frac{\omega^2\bar{\lambda}_L}{2})\bp}\\
	&~~~+\Vert\nxk\Vert^2_{(\frac{1}{\eta\omega}+\frac{1}{2})\bp+\frac{1}{2}\bk+4\tilde{b}_2L_f^2\bi_{nd}}\!+\!\Vert\nvk\Vert^2_{\frac{1}{\eta}\bp^2+\bp+\eta\omega\bp}\\
	&~~~+(\frac{\eta}{4}+4\tilde{b}_2\eta^2\omega^2L_f^2)\Vert\sum_{t=0}^{k}\bar{\xi}_{v,t}\Vert^2,\addtag\label{eq:upperboundofxp1}
\end{align*}
where the first and second equalities comes from~\eqref{eq:iterationxp3} and~\eqref{eq:iterationv3}; the first inequality comes from~\eqref{eq:inequality1} and~\eqref{eq:propertyofp} and the fact that $\bk=\bi-\bh$; the second and third inequalities holds due to~\eqref{eq:inequality1}, \eqref{eq:inequality2},~\eqref{eq:propertyofp}, and~$\lb_\bk=1$; the last inequality comes from~\eqref{eq:propertyofp1},~\eqref{eq:propertyofbg1}, and~\eqref{eq:propertyofbg3}; $\tilde{b}_2=\frac{1+2\eta}{2\eta\omega^2\underline{\lambda}_L^2}+\frac{1}{4}$. 

We then provide the upper bound of $n(f(\bar{x}_{k+1}-f^*))$.
\begin{align*}
	n(f(\bar{x}_{k+1})&-f^*)=\tilde{f}(\bx_k)-nf^*+\tilde{f}(\bx_{k+1})-\tilde{f}(\bx_k)\\
	&\leq \tilde{f}(\bx_k)-nf^*-\eta(\bg_k+\omega\sum_{t=0}^{k}\bar{\xi}_{v,t}-\frac{\nbxk}{\eta})^\top\bg_k^b\\
	&~~~+\frac{L_f}{2}\Vert\nbxk-\eta\bg_k-\eta\omega\sum_{t=0}^{k}\bar{\xi}_{v,t}\Vert^2\\
	&\leq\tilde{f}(\bx_k)-nf^*-\frac{\eta}{2}\Vert\bg_k\Vert^2-\frac{\eta}{2}\Vert\bg_k^b\Vert^2\\
	&~~~+\frac{\eta}{2}\Vert\bg_k-\bg_k^b\Vert^2-\eta(\omega\sum_{t=0}^{k}\bar{\xi}_{v,t}-\frac{\nbxk}{\eta})^\top\bg_k^b\\
	&~~~+\eta^2L_f\Vert\bg_k\Vert^2+2\eta^2L_f\omega^2\Vert\sum_{t=0}^{k}\bar{\xi}_{v,t}\Vert^2\\
	&~~~+2L_f\Vert\nxk\Vert^2\\
	&\leq\tilde{f}(\bx_k)-nf^*-\frac{\eta}{2}(1-2\eta L_f)\Vert\bg_k\Vert^2-\frac{\eta}{4}\Vert\bg_k^b\Vert^2\\
	&~~~+\Vert\x_k\Vert^2_{\frac{\eta}{2}L_f^2\bk}+2\eta^2\omega^2(\frac{1}{\eta}+L_f)\Vert\sum_{t=0}^{k}\bar{\xi}_{v,t}\Vert^2\\
	&~~~+2(\frac{1}{\eta}+L_f)\Vert\nxk\Vert^2,\addtag\label{eq:upperboundofax1}
\end{align*}
where the first inequality comes from~\eqref{eq:propertyofbx}, Assumption~\ref{as:smooth}, and the fact that $\bh=\bh\bh$; the second and third inequalities hold due to~\eqref{eq:inequality2}; the last inequality comes from~\eqref{eq:inequality1},~\eqref{eq:inequality2}, and~\eqref{eq:propertyofbg2}.
\begin{align*}
\Vert\x_{k+1}+&\nxkk-\x^c_{k+1}\Vert^2\\
&=\Vert\x_{k+1}+\nxkk-\x_k-\nxk+\x_k+\nxk\\
&~~~-\x_{k}^c-\alpha_xr\frac{C}{r}(\x_k+\nxk-\x_{k}^c)\Vert^2\\
&\leq (1+s)(\alpha_xr(1-\varphi)+(1-\alpha_xr))\Vert\x_k\\
&~~~+\nxk-\x_{k}^c\Vert^2\\
&~~~+(1+\frac{1}{s})\Vert\x_{k+1}+\nxkk-\x_k-\nxk\Vert^2\\
&\leq (1-\varphi_2-\frac{\varphi_2^2}{2})\Vert\x_k+\nxk-\x_{k}^c\Vert^2\\
&~~~+(1+\frac{2}{\varphi_2})\Vert\x_{k+1}+\nxkk-\x_k-\nxk\Vert^2,\addtag\label{eq:upperboundofc1}
\end{align*}
where the first equality comes from~\eqref{citerationx},~\eqref{citerationxc}; the first inequality comes from~\eqref{eq:inequality1}; the second inequality follows by denoting $\varphi_2=\alpha_xr\varphi$, choosing $s=\frac{\varphi_2}{2}$, and $\alpha_xr<1$. We have
\begin{align*}
	\Vert\x_{k+1}&+\nxkk-\x_k-\nxk\Vert^2\\
	&=\Vert\eta(\gamma\bl(\hx_k-\x_k-\nxk)+\gamma\bl\x_k+\omega\vv_k+\g_k^b\\
	&~~~+\g_k-\g_k^b)+\nxkk+(\eta\gamma\bl-\bi)\nxk\Vert^2\\
	&\leq 6\eta^2(\Vert\gamma\bl(\hx_k-\x_k-\nxk)\Vert^2+\Vert\omega\vv_k+\g_k^b\Vert^2\\
	&~~~+\Vert\gamma\bl\x_k\Vert^2+\Vert\g_k-\g_k^b\Vert^2)+6\Vert\nxkk\Vert^2\\
	&~~~+6\Vert(\eta\gamma\bl-\bi)\nxk\Vert^2\\
	&\leq 6\eta^2(\gamma^2\bar{\lambda}_L^2r_0\Vert\x^c_k-\x_k-\nxk\Vert^2\!+\!\Vert\vv_k\!+\!\frac{1}{\omega}\g_k^b\Vert^2_{\omega^2\underline{\lambda}_L\bp}\\
	&~~~+\Vert\x_k\Vert^2_{(\gamma^2\bar{\lambda}_L^2+L_f^2)\bk})+6\Vert\nxkk\Vert^2\\
	&~~~+6\bar{\lambda}^2_{\eta\gamma\bl-\bi}\Vert\nxk\Vert^2,\addtag\label{eq:upperboundofc2}
\end{align*}
where the first equality holds due to~\eqref{eq:iterationxp3}; the first inequality holds due to Jensen's inequality; the last inequality holds due to~\eqref{eq:propertyofcompressors1},~\eqref{citerationx},~\eqref{eq:propertyofk1},~\eqref{eq:propertyofp1}, and~\eqref{eq:propertyofbg2}. Combining~\eqref{eq:upperboundofc1}--\eqref{eq:upperboundofc2}, one obtains that
\begin{align*}
	\Vert\x_{k+1}+&\nxkk-\x^c_{k+1}\Vert^2\\
	&\leq (1-\frac{\varphi_2}{2}-\frac{\varphi_2^2}{2}+6\eta^2\gamma^2\bar{\lambda}_L^2r_0(1+\frac{2}{\varphi_2}))\Vert\x_k\\
	&~~~+\nxk-\x_{k}^c\Vert^2+\Vert\x_k\Vert^2_{6\eta^2(1+\frac{2}{\varphi_2})(\gamma^2\bar{\lambda}_L^2+L_f^2)\bk}\\
	&~~~+\Vert\vv_k+\frac{1}{\omega}\g_k^b\Vert^2_{6\eta^2(1+\frac{2}{\varphi_2})\omega^2\underline{\lambda}_L\bp}\\
	&~~~+6(1+\frac{2}{\varphi_2})\Vert\nxkk\Vert^2\\
	&~~~+6\bar{\lambda}^2_{\eta\gamma\bl-\bi}(1+\frac{2}{\varphi_2})\Vert\nxk\Vert^2,\addtag\label{eq:upperboundofc3}
\end{align*}
From~$\bar{\lambda}_{\bk}=1$,~\eqref{eq:propertyofp1},~\eqref{eq:upperboundofx1}--\eqref{eq:upperboundofax1}, and~\eqref{eq:upperboundofc3}, we have
\begin{align*}
	\widetilde{V} _{k+1}&\leq\widetilde{V} _{k}-\Vert\x_k\Vert^2_{B_1\eta}+\Vert\hx_k\Vert^2_{B_2\eta}\\
	&~~~-\Vert\vv_k+\frac{1}{\omega}\g_k^b\Vert^2_{(\kappa_3-\kappa_4\eta)\eta\bp}-(\kappa_5-\kappa_6\eta)\eta \Vert\bg_k\Vert^2\\
	&~~~-\tilde{b}_3\Vert\x_{k}+\nxk-\x^c_{k}\Vert^2-\frac{\eta}{4}\Vert\bg_k^b\Vert^2+\tilde{b}_4\Vert\nxk\Vert^2\\
	&~~~+\kappa_{11}\Vert\nvk\Vert^2+\kappa_{12}\Vert\sum_{t=0}^{k}\bar{\xi}_{v,t}\Vert^2,\addtag\label{eq:lyapunov}
\end{align*}
where
\begin{align*}
	&B_1=\frac{\gamma\bl}{2}-\frac{1}{4}(7+9L_f^2+\omega)\bk-\eta(\frac{5L_f^2}{2}\\
	&~~~~~~+6(1+\frac{2}{\varphi_2})(\gamma^2\bar{\lambda}_L^2+L_f^2)\bk),\\
	&B_2=\omega\bk+\eta(2\gamma^2\bl^2+\frac{3\omega^2+\gamma^2}{2}\bl+(\omega^2+\frac{1}{2})\bk),\\
	&\tilde{b}_3=\frac{\varphi_2}{2}+\frac{\varphi_2^2}{2}-6\eta^2\gamma^2\bar{\lambda}_L^2r_0(1+\frac{2}{\varphi_2})-\frac{\eta}{2}(\gamma+4\omega)\bar{\lambda}_L,\\
	&\tilde{b}_4=4(\frac{L_f^2}{\omega^2\underline{\lambda}_L}+\frac{\gamma L_f^2}{\omega^3\underline{\lambda}_L}+\frac{L_f^2}{2\eta\omega\underline{\lambda}_L}+\frac{\gamma L_f^2}{2\eta\omega^2\underline{\lambda}_L}\\
	&~~~~~~+\frac{(1+2\eta)L_f^2}{2\eta\omega^2\underline{\lambda}_L^2}+\frac{3L_f^2}{4}+\frac{1}{2\eta}+\frac{L_f}{2})+\frac{\bar{\lambda}_{\bi_{nd}+\eta\gamma\bl}^2}{\eta}\\
	&~~~~~~+10+\frac{\eta}{2}+2\eta\omega\bar{\lambda}_L+(\frac{1}{\eta\omega}+\frac{1}{2})\frac{1}{\underline{\lambda}_L}+\frac{12}{\alpha_x r\varphi}\\
	&~~~~~~+6\bar{\lambda}^2_{\eta\gamma\bl-\bi}(1+\frac{2}{\varphi_2}).
\end{align*}
Since Jensen's inequality and~$\bar{\lambda}_{\bk}=1$, it holds that
\begin{align*}
	\Vert\hx_k\Vert^2_{\bk}&=\Vert\hx_k-\x_k-\nxk+\x_k+\nxk\Vert^2_{\bk}\\
	&\leq 3r_0\Vert\x_k+\nxk-\x^c_k\Vert^2+3\Vert\x_k\Vert^2_{\bk}+3\Vert\nxk\Vert^2.\addtag\label{eq:propertyofhx}
\end{align*}
Combining~\eqref{eq:lyapunov}--\eqref{eq:propertyofhx}, we have~\eqref{eq:lyapunov1}. Then the proof is completed.

\subsection{The proof of Theorem~\ref{theo:convergence1}}
For simplicity of the proof, we also denote some notations and a useful auxiliary function
\begin{align*}
&\tilde{\zeta}_1\geq\max\{\frac{13+\tilde{\zeta}_4}{2\underline{\lambda}_L},1\},~\tilde{\zeta}_2=\max\{\tilde{\zeta}_4,\tilde{\zeta}_5\},\\
&\tilde{\zeta}_3=\min\{\frac{\kappa_1}{\kappa_2},\frac{\kappa_3}{\kappa_4},\frac{\kappa_5}{\kappa_6},\frac{\sqrt{\kappa_8^2+4\kappa_7\kappa_9}-\kappa_8}{2\kappa_9}\},\\
&\tilde{\zeta}_4=\frac{6}{\underline{\lambda}_L}+1+\tilde{\zeta}_1+2\sqrt{(\frac{3}{\underline{\lambda}_L}+\frac{1}{2}(1+\tilde{\zeta}_1))^2+\frac{1}{2\omega}(1+\tilde{\zeta}_1)},\\
&\tilde{\zeta}_5=2\frac{L_f^2+\tilde{\zeta}_1L_f^2}{\underline{\lambda}_L}+2\sqrt{(\frac{L_f^2+\tilde{\zeta}_1L_f^2}{\underline{\lambda}_L})^2+\frac{L_f^2}{\omega^2\underline{\lambda}_L^2}},\\
&\kappa_{13}=\frac{\gamma \underline{\lambda}_L-\omega}{2 \gamma \underline{\lambda}_L}\\
&\bar{\kappa}_3=\max\{\frac{4}{\eta},\frac{1}{(\kappa_1-\kappa_2\eta)\eta}\},\\
&M_2=\frac{1}{n}\mathbb{E}[\widetilde{V}_{0}]+(\kappa_{10}+\kappa_{11})\sum_{k=0}^{\infty}(2d\bar{q}^{2k}\bar{s}_\xi^2),\\
&\widetilde{U}_k=\Vert\x_{k}\Vert_\bk^2+\Vert\vv_{k}+\frac{1}{\omega}\g_{k}^b\Vert^2_{\bp}+\Vert\x_{k}+\nxk-\x^c_{k}\Vert^2\\
	&~~~~~~+n(f(\bar{x}_{k})-f^*).
\end{align*}
From~\eqref{eq:inequality1}, we have
\begin{align*}
  \widetilde{V}_k&\geq \frac{1}{2}\Vert\x_{k}\Vert_\bk^2+\frac{1}{2}(1+\frac{\gamma}{\omega})\Vert\vv_{k}+\frac{1}{\omega}\g_{k}^b\Vert^2_{\bp}-\frac{\omega}{2\gamma\underline{\lambda}_L}\Vert\x_{k}\Vert_\bk^2\\
  &~~~-\frac{\gamma}{2\omega}\Vert\vv_{k}+\frac{1}{\omega}\g_{k}^b\Vert^2_{\bp}+\Vert\x_{k}+\nxk-\x^c_{k}\Vert^2\\
  &~~~+n(f(\bar{x}_{k})-f^*)\\
  &\geq \kappa_{13}\widetilde{U}_k\geq0,\addtag\label{eq:relationbuv1}
\end{align*}
We then verify $\kappa_1-\kappa_2\eta,~\kappa_3-\kappa_4\eta$, $\kappa_5-\kappa_6\eta$, $\kappa_7-\kappa_8\eta-\kappa_9\eta^2$ are positive. Since $\gamma=\tilde{\zeta}_1\omega$, $\tilde{\zeta}_1\geq\frac{13+\tilde{\zeta}_4}{2\underline{\lambda}_L}$, $\tilde{\zeta}_4>0$, and $\omega>\tilde{\zeta}_2>\frac{7+9L_f^2}{\tilde{\zeta}_4}$, it holds that
\begin{align*}
	\kappa_1>\frac{\tilde{\zeta}_1\omega\underline{\lambda}_L}{2}-\frac{1}{4}(\tilde{\zeta}_2\omega+13\omega)>0.
\end{align*}
From~$\gamma=\tilde{\zeta}_1\omega$,~$\tilde{\zeta}_2>\tilde{\zeta}_4$ and~$\omega>0$, we have
\begin{align*}
	\kappa_3=\frac{\omega}{4}-\frac{3}{\underline{\lambda}_L}-\frac{1}{2}(1+\tilde{\zeta}_1)-\frac{1}{2\omega}(1+\tilde{\zeta}_1)>0.
\end{align*}
From~$\gamma=\tilde{\zeta}_1\omega$,~$\tilde{\zeta}_2>\tilde{\zeta}_5$ and~$\omega>0$, we have
\begin{align*}
	\kappa_5=\frac{1}{4}-\frac{1}{\omega}\frac{L_f^2+\tilde{\zeta}_1L_f^2}{\underline{\lambda}_L}-\frac{L_f^2}{\omega^2\underline{\lambda}_L^2}>0.
\end{align*}
From~$0<\eta<\tilde{\zeta}_3$, we can verify $\kappa_1-\kappa_2\eta,~\kappa_3-\kappa_4\eta$, $\kappa_5-\kappa_6\eta$, $\kappa_7-\kappa_8\eta-\kappa_9\eta^2$ are positive. From Lemma~\ref{lemma:lyapunov}, we have
\begin{align*}
	\widetilde{V} _{k+1}&\leq\widetilde{V}_{k}-\Vert\x_k\Vert^2_{(\kappa_1-\kappa_2\eta)\eta\bk}-\frac{\eta}{4}\Vert\bg_k^b\Vert^2+\kappa_{10}\Vert\nxk\Vert^2\\
&~~~+\kappa_{11}\Vert\nvk\Vert^2+\kappa_{12}\Vert\sum_{t=0}^{k}\bar{\xi}_{v,t}\Vert^2.
\end{align*}
Then, one obtains that
\begin{align*}
\sum_{k=0}^T\big((\kappa_1-\kappa_2\eta)\eta\Vert\x_k\Vert^2_{\bk}&+\frac{\eta}{4}\Vert\bg_k^b\Vert^2\big)\leq \widetilde{V}_{0}+\sum_{k=0}^T\big(\kappa_{10}\Vert\nxk\Vert^2\\
&~~~+\kappa_{11}\Vert\nvk\Vert^2+\kappa_{12}\Vert\sum_{t=0}^{k}\bar{\xi}_{v,t}\Vert^2\big),
\end{align*}
which can be rewritten as 
\begin{align*}
    \frac{1}{T}\sum_{k=0}^T\big(\mathbb{E}\Vert\x_k-\bx_k\Vert^2&+\mathbb{E}\Vert \nabla f(\bar{x}_k\Vert^2\big)\leq \frac{\bar{\kappa}_3M_2}{T}\\
&+\frac{\bar{\kappa}_3\kappa_{12}}{n}\mathbb{E}\Vert\sum_{t=0}^{\infty}\bar{\xi}_{v,t}\Vert^2.
\end{align*}


\section{The proof of Theorem~\ref{theo:convergence2}}\label{app-convergence2}
In this proof, in addition to the notations used in the proof of Theorem~\ref{theo:convergence1}, we also denote
\begin{align*}
    &\kappa_{14}=\max\{\frac{1}{2}+\frac{\gamma}{\omega},\frac{\gamma\underline{\lambda}_L+\omega}{2\gamma\underline{\lambda}_L}\},\\
&\kappa_{15}=\eta\min\{\kappa_1-\kappa_2\eta,~\kappa_3-\kappa_4\eta,\frac{\nu}{2},\frac{\kappa_7}{\eta}-\kappa_8-\kappa_9\eta\}\\
&0<\bar{\kappa}_4<\min\{\frac{\kappa_{15}}{\kappa_{14}},1-\bar{q}^2\}.
\end{align*}

From Lemma~\ref{lemma:lyapunov} and Assumption~\ref{as:PLcondition}, one obtains that
\begin{align*}
	\widetilde{V} _{k+1}&\leq\widetilde{V} _{k}-\Vert\x_k\Vert^2_{(\kappa_1-\kappa_2\eta)\eta\bk}-\Vert\vv_k+\frac{1}{\omega}\g_k^b\Vert^2_{(\kappa_3-\kappa_4\eta)\eta\bp}\\
	&~~~-(\kappa_5-\kappa_6\eta)\eta \Vert\bg_k\Vert^2-(\kappa_7-\kappa_8\eta-\kappa_9^2\eta^2)\Vert\x_{k}\\
	&~~~+\nxk-\x^c_{k}\Vert^2-\frac{\eta\nu n}{2}(f(\bar{x}_{k})-f^*)+\kappa_{10}\Vert\nxk\Vert^2\\
	&~~~+\kappa_{11}\Vert\nvk\Vert^2+\kappa_{12}\Vert\sum_{t=0}^{k}\bar{\xi}_{v,t}\Vert^2.\addtag\label{eq:lyapunov2}
\end{align*}
From~\eqref{eq:inequality1}, we have
\begin{align*}
  \widetilde{V}_k&\leq \frac{1}{2}\Vert\x_{k}\Vert_\bk^2+\frac{1}{2}(1+\frac{\gamma}{\omega})\Vert\vv_{k}+\frac{1}{\omega}\g_{k}^b\Vert^2_{\bp}+\frac{\omega}{2\gamma\underline{\lambda}_L}\Vert\x_{k}\Vert_\bk^2\\
  &~~~+\frac{\gamma}{2\omega}\Vert\vv_{k}+\frac{1}{\omega}\g_{k}^b\Vert^2_{\bp}+\Vert\x_{k}+\nxk-\x^c_{k}\Vert^2\\
  &~~~+n(f(\bar{x}_{k})-f^*)\\
  &\leq \kappa_{14}\widetilde{U}_k,\addtag\label{eq:relationbuv}
\end{align*}
Combining~\eqref{eq:lyapunov2} and~\eqref{eq:relationbuv}, we have
\begin{align*}
	\E[\widetilde{V} _{k+1}]&\leq\E[\widetilde{V} _{k}]-\frac{\kappa_{15}}{\kappa_{14}}\E[\widetilde{V} _{k}]+\kappa_{10}\E[\Vert\nxk\Vert^2]+\kappa_{11}\E[\Vert\nvk\Vert^2]\\
	&~~~+\kappa_{12}\E[\Vert\sum_{t=0}^{k}\bar{\xi}_{v,t}\Vert^2]\\
	&\leq(1-\bar{\kappa}_4)^{k+1}\E[\widetilde{V}_0]\\
 &+(\kappa_{10}+\kappa_{11})n^2\bar{s_\xi}^2\sum_{t=0}^k(1-\bar{\kappa}_4)^{k-t}\bar{q}^{2t}\\
	&~~~+\kappa_{12}\sum_{t=0}^k(1-\bar{\kappa}_4)^{k-t}\E[\Vert\sum_{m=0}^{t}\bar{\xi}_{v,m}\Vert^2]\\
	&\leq(1-\bar{\kappa}_4)^{k+1}(\E[\widetilde{V}_0]+\frac{(\kappa_{10}+\kappa_{11})n^2\bar{s_\xi}^2}{1-\bar{\kappa}_4-\bar{q}^2})\\
    &~~~+\kappa_{12}\frac{1}{\bar{\kappa}}\E[\Vert\sum_{k=0}^{\infty}\xi_{v,k}\Vert^2].\addtag\label{eq:relationbuv2}
\end{align*}
Noting that $0<\bar{\kappa}_4<1$ since $\bar{q}<1$ and $\frac{\kappa_{15}}{\kappa_{14}}\leq\frac{2\kappa_{7}}{3}=\frac{\alpha_xr\varphi+(\alpha_xr\varphi)^2}{3}<1$. Combining~\eqref{eq:relationbuv1} and~\eqref{eq:relationbuv2}, we complete the proof.

\section{The proof of Theorem~\ref{theo:privacy1}}\label{app-privacy1}
Similar to the proof of Theorem~\ref{theo:privacy}, we know that the Theorem~\ref{theo:privacy1} can be proved if the following inequality holds for any observation $\mathcal{H}\subseteq\text{Range}(C)$ and any pair of adjacent cost function sets $\mathcal{S}^{(1)}$ and $\mathcal{S}^{(2)}$,
\begin{align*}
P\{(\xi_{x},\xi_{v},&\mathbf{\varrho})\in\Psi|Z_{\mathcal{F}^{(1)}}(\xi_{x},\xi_{v},\mathbf{\varrho})\in\mathcal{H}\}\\
&\leq e^\epsilon P\{(\xi_{x},\xi_{v},\mathbf{\varrho})\in\Psi|Z_{\mathcal{F}^{(2)}}(\xi_{x},\xi_{v},\mathbf{\varrho})\in\mathcal{H}\},
\end{align*}
where $\mathcal{F}^{(l)}\!\!=\!\!\{\x(0),W,\mathcal{S}^{(l)}\!\}$, $l\!=\!1,2$, and $\Psi$ denotes the sample space. Then it is indispensable to guarantee $Z_{\mathcal{F}^{(1)}}(\xi_{x},\xi_{y},\mathbf{\varrho})=$ $Z_{\mathcal{F}^{(2)}}(\xi_{x},\xi_{y},\mathbf{\varrho})$, i.e.,
\begin{align*}
C(x_{i,k}^{a-c,(1)},\varrho_k)=C(x_{i,k}^{a-c,(2)},\varrho_k),
\end{align*}
for $\forall i\in\mathcal{V}$ and any $k\geq0$, where
\begin{align*}
&x_{i,k}^{a-c,(l)}=x_{i,k}^{a,(l)}-x_{i,k-1}^{c,(l)},~l=1,2.
\end{align*}
Similar to~\eqref{eq:ineqofprivacy}, we have
\begin{align}\label{eq:ineqofprivacy1}
\begin{aligned}
&\frac{P\{(\xi_{x},\xi_{v},\mathbf{\varrho})\in\Psi|Z_{\mathcal{F}^{(1)}}(\xi_{x},\xi_{v},\mathbf{\varrho})\in\mathcal{H}\}}{P\{(\xi_{x},\xi_{v},\mathbf{\varrho})\in\Psi|Z_{\mathcal{F}^{(2)}}(\xi_{x},\xi_{v},\mathbf{\varrho})\in\mathcal{H}\}}\\
&\leq\frac{P\{(\xi_{x},\xi_{v},\mathbf{\varrho})\in\Psi|Z_{\mathcal{F}^{(1)}}(\xi_{x},\xi_{v},\mathbf{\varrho})\in\mathcal{H}\}}{P\{(\xi_{x},\xi_{v},\mathbf{\varrho})\in\Psi|Z_{\mathcal{F}^{(2)}}(\xi_{x},\xi_{v},\mathbf{\varrho})\in\mathcal{H},E_2\}},
\end{aligned}
\end{align}	
where $E_2=\cup_{k=0}^\infty\{x_{i_0,k}^{a,(2)}=x_{i_0,k}^{a,(1)}\}$ is an event. From~\eqref{eq:iterationxp3} and~\eqref{eq:iterationv3}, we have 
\begin{align*}
    &~x_{i_0,k+1}^{(1)}-x_{i_0,k+1}^{(2)}=-\eta(\omega  v_{i_0,k}^{(1)}-\omega  v_{i_0,k}^{(2)}+\nabla f_{i_0}^{(1)}(x_{i_0,k}^{(1)})\\&~~~~~~~~~~~~~~~~~~~~~~~~~~~~~-\nabla f_{i_0}^{(2)}(x_{i_0,k}^{(2)})),\addtag\label{eq:iterationxp4}\\
	&~v_{i_0,k+1}^{(1)}-v_{i_0,k+1}^{(2)}=v_{i_0,k}^{(1)}-v_{i_0,k}^{(2)}+\xi_{v_{i_0},k}^{(1)}-\xi_{v_{i_0},k}^{(2)}.\addtag\label{eq:iterationv4}\\
\end{align*}
We then denote the following map by $\mathcal{B}(\cdot)$, i.e., $(\xi_{x_{i_0}}^{(2)},\xi_{v_{i_0}}^{(2)})=\mathcal{B}(\xi_{x_{i_0}}^{(1)},\xi_{v_{i_0}}^{(1)})$.
\begin{align*}
    &\xi_{x_{i_0},0}^{(2)}=\xi_{x_{i_0},0}^{(1)},\\
    &\xi_{x_{i_0},1}^{(2)}=\xi_{x_{i_0},1}^{(1)}-\eta(\nabla f_{i_0}^{(1)}(x_{i_0,0}^{(1)})-\nabla f_{i_0}^{(2)}(x_{i_0,0}^{(2)})),\\
    &\xi_{x_{i_0},k+1}^{(2)}=\xi_{x_{i_0},k+1}^{(1)}-\eta(\nabla f_{i_0}^{(1)}(x_{i_0,k}^{(1)})-\nabla f_{i_0}^{(2)}(x_{i_0,k}^{(2)})\\
    &~~~~~~~~~~~~-\nabla f_{i_0}^{(1)}(x_{i_0,k-1}^{(1)})+\nabla f_{i_0}^{(2)}(x_{i_0,k-1}^{(2)})), \forall k\geq 1,\\
    &\xi_{v_{i_0},0}^{(2)}=\xi_{v_{i_0},0}^{(1)},\\
    &\xi_{v_{i_0},1}^{(2)}=\xi_{v_{i_0},1}^{(1)}+\frac{1}{\omega}(\nabla f_{i_0}^{(1)}(x_{i_0,0}^{(1)})-\nabla f_{i_0}^{(2)}(x_{i_0,0}^{(2)})),\\
    &\xi_{v_{i_0},k+1}^{(2)}=\xi_{v_{i_0},k+1}^{(1)}+\frac{1}{\omega}(\nabla f_{i_0}^{(1)}(\nabla f_{i_0}^{(1)}(x_{i_0,k}^{(1)})-\nabla f_{i_0}^{(2)}(x_{i_0,k}^{(2)})\\
    &~~~~~~~~~~~~-\nabla f_{i_0}^{(1)}(x_{i_0,k-1}^{(1)})+\nabla f_{i_0}^{(2)}(x_{i_0,k-1}^{(2)})), \forall k\geq 1.
\end{align*}
From~\eqref{eq:iterationxp4},~\eqref{eq:iterationv4}, and~$\mathcal{B}(\cdot)$, it is easy to verify that $x_{i_0,k}^{a,(2)}=x_{i_0,k}^{a,(1)},~\forall k\geq0$ holds. Then combining~\eqref{eq:ineqofprivacy1}, we have
\begin{align}\label{eq:ineqofprivacy2}
\begin{aligned}
&\frac{P\{(\xi_{x},\xi_{v},\mathbf{\varrho})\in\Psi|Z_{\mathcal{F}^{(1)}}(\xi_{x},\xi_{v},\mathbf{\varrho})\in\mathcal{H}\}}{P\{(\xi_{x},\xi_{v},\mathbf{\varrho})\in\Psi|Z_{\mathcal{F}^{(2)}}(\xi_{x},\xi_{v},\mathbf{\varrho})\in\mathcal{H}\}}\\
&\leq\frac{P\{(\xi_{x},\xi_{v},\mathbf{\varrho})\in\Psi|Z_{\mathcal{F}^{(1)}}(\xi_{x},\xi_{v},\mathbf{\varrho})\in\mathcal{H}\}}{P\{(\xi_{x},\xi_{v},\mathbf{\varrho})\in\Psi|Z_{\mathcal{F}^{(2)}}(\mathcal{B}(\xi_{x},\xi_{v}),\mathbf{\varrho})\in\mathcal{H}\},E_2}.
\end{aligned}
\end{align}	
Thus, from~\eqref{eq:ineqofprivacy2}, the proof can be completed in the same way as the proof of Theorem~\ref{theo:privacy}.

\bibliographystyle{IEEEtran}
\bibliography{ref_Antai}

\end{document}